\documentclass[12pt,  a4paper]{article}

\usepackage{amssymb,amsfonts}
\usepackage{mathtools}
\usepackage{amsmath,amsthm}
\usepackage{mathrsfs}

\usepackage{exscale,txfonts}

\usepackage[all]{xy}
\usepackage[active]{srcltx}
\usepackage{epic,eepic,eepicemu}
\usepackage[usenames,dvipsnames]{xcolor}
\usepackage[normalem]{ulem}
\usepackage{enumerate}
\usepackage[colorlinks,linkcolor=BrickRed,citecolor=OliveGreen,urlcolor=black,hypertexnames=true]{hyperref}

\allowdisplaybreaks

\def\titlerunning#1{\gdef\titrun{#1}}
\makeatletter
\def\author#1{\gdef\autrun{\def\and{\unskip, }#1}\gdef\@author{#1}}
\def\address#1{{\def\and{\\\hspace*{18pt}}\renewcommand{\thefootnote}{}%
\footnote {#1}}%
\markboth{\autrun}{\titrun}}
\makeatother
\def\email#1{e-mail: #1}
\def\subjclass#1{{\renewcommand{\thefootnote}{}%
\footnote{\emph{Mathematics Subject Classification (2010):} #1}}}
\def\keywords#1{\par\medskip
\noindent\textbf{Keywords.} #1}

\numberwithin{equation}{section}

\definecolor{dgreen}{rgb}{0.1,0.7,0.1}
\definecolor{purple}{rgb}{0.49, 0.06, 0.51}

\def\N{\mathbb{N}}

\def\Z{\mathbb{Z}}

\def\CF{{\cal F}}

\def\CT{{\cal T}}
\newcommand{\CC}{\mathscr{C}}
\newcommand{\CP}{\mathscr{P}}
\newcommand{\CQ}{\mathscr{Q}}

\newcommand{\CM}{\mathscr{M}}
\newcommand{\cB}{\mathcal{B}}
\newcommand{\mm}{\mathrm{max}}

\def\la{\langle}
\def\ra{\rangle}
\def\s{\sigma}

\newcommand{\patch}{\textup{\textrm{patch}}}

\newcommand{\vt}{\vartheta}
\newcommand{\ox}{\otimes}
\newcommand{\x}{\times}
\newcommand{\sm}{\setminus}
\newcommand{\smz}{\setminus \{0\}}
\newcommand{\Nil}{\mathrm{Nil}}
\newcommand{\wt}{\widetilde}
\newcommand{\id}{\mathrm{id}}
\newcommand{\vf}{\varphi}

\newcommand{\too}{\longrightarrow}

\newcommand{\bbar}{\overline{\phantom{x}}}

\DeclareMathOperator{\im}{Im}

\DeclareMathOperator{\sign}{sign}

\DeclareMathOperator{\Sym}{Sym}
\DeclareMathOperator{\Skew}{Skew}
\DeclareMathOperator{\rk}{rank}

\DeclareMathOperator{\tr}{tr}

\DeclareMathOperator{\Int}{Int}

\DeclareMathOperator{\Trd}{Trd}

\DeclareMathOperator{\Tr}{\mathrm{Tr}}
\DeclareMathOperator{\diag}{diag}
\DeclareMathOperator{\Diag}{Diag}

\DeclareMathOperator{\GL}{GL}
\DeclareMathOperator{\PSD}{PSD}
\DeclareMathOperator{\Span}{Span}

\renewcommand{\geq}{\geqslant}
\renewcommand{\leq}{\leqslant}

\renewcommand{\ge}{\geqslant}
\renewcommand{\le}{\leqslant}

\newcommand{\swap}{\widehat{\phantom{x}}}

\newcommand{\df}{\emph}
\newcommand{\ve}{\varepsilon}

\newcommand{\ad}{\mathrm{ad}}

\newcommand{\das}{D_{(A,\s)}}

\newcommand{\pf}[1]{\langle\!\langle #1\rangle\!\rangle}

\newcommand{\qf}[1]{\langle #1\rangle}

\newcommand{\ns}{\mathrm{ns}}
\newcommand{\qnd}[1]{\langle #1\rangle^{\ns}}

\newcommand{\mor}{\mathfrak{m}}

\DeclareMathOperator{\Herm}{\mathfrak{Herm}}

\newcommand{\tu}{\textup}

\newtheorem{thm}{Theorem}[section]
\newtheorem{prop}[thm]{Proposition}
\newtheorem{cor}[thm]{Corollary}
\newtheorem{lemma}[thm]{Lemma}
\theoremstyle{definition}
\newtheorem{defi}[thm]{Definition}
\newtheorem{fact}[thm]{Fact}
\newtheorem{remark}[thm]{Remark}
\newtheorem{ex}[thm]{Example}

\begin{document}

\titlerunning{Positive cones on algebras with involution}

\title{Positive cones on algebras with involution\footnote{Accepted manuscript \copyright 2019,
made available under the  CC-BY-NC-ND 4.0 license 
\url{http://creativecommons.org/licenses/by-nc-nd/4.0/}.\newline 
Published article: Advances in Mathematics 361 (2020) 106954.\newline 
DOI: \href{https://doi.org/10.1016/j.aim.2019.106954}{10.1016/j.aim.2019.106954}}}

\author{Vincent Astier
\and 
Thomas Unger}

\date{}


\maketitle

\centerline{\footnotesize{\emph{Dedicated to Manfred Knebusch on the occasion of his eightieth birthday}}}

\address{V. Astier, T. Unger (Corresponding author):
School of Mathematics and Statistics, University College Dublin, Belfield, Dublin~4, Ireland;
\email{vincent.astier@ucd.ie, thomas.unger@ucd.ie}}


\begin{abstract}
We introduce positive cones on algebras with involution. These allow us to prove analogues of
Artin's solution to Hilbert's 17th problem, the Artin-Schreier theorem characterizing formally real
fields, and to define signatures with respect to positive cones. We consider the space of positive
cones of an algebra with involution and investigate its topological properties, showing in
particular that it is a spectral space. As an application we solve the problem of the existence of
positive involutions.


\end{abstract}

\tableofcontents


\section{Introduction}

In a series of papers \cite{A-U-Kneb}, \cite{A-U-prime}, \cite{A-U-PS}, \cite{A-U-stab} we initiated an investigation of 
central simple algebras with involution from a real algebraic point of view,
inspired by the classical correspondences between
signatures of quadratic forms over a field $F$, morphisms from the Witt ring $W(F)$ into $\Z$,  
prime ideals of  $W(F)$, and orderings on  $F$,
which form one of the foundations of real algebra.

More precisely, in \cite{A-U-Kneb} we defined signatures of hermitian forms over   algebras with involution $(A,\s)$,  
and in \cite{A-U-prime} we showed 
that these provide  the desired natural correspondences with morphisms from the Witt group $W(A,\s)$ into $\Z$ and with ``prime
ideals'' of $W(A,\s)$. 

In the present paper we show that these correspondences can be extended to include a notion of partial ordering on 
algebras with involution, which we call (pre-)positive cone (Definition~\ref{def-preordering}). 
Prepositive cones  are inspired by both positive semidefinite
matrices and one-dimensional hermitian forms of maximal signature, as well as by Prestel's notion of quadratic semiorderings \cite{Prestel73},
and can  be interpreted as those orderings on the 
base field that extend to the algebra (Proposition~\ref{easy}). In addition to the above correspondences,  positive cones 
allow us to obtain analogues of the Artin-Schreier theorem on the 
characterization of formally real fields (Theorem~\ref{fr}) and 
Artin's solution of Hilbert's 17th problem (Theorem~\ref{intersection}, Corollary~\ref{Artin}). 
We also define signatures with respect to positive cones (Section~\ref{subsec:signp}) and
show that the space of  all positive cones  is a spectral space (Theorem~\ref{thm:spectral}),  whose topology is linked to the topology of the space of orderings of the base field (Proposition~\ref{pi-homeom} and Corollary~\ref{double}). 

Finally, as an application, we answer the question of the existence of positive involutions (Theorem~\ref{main_pos}),  a question
that does not seem to have been treated before in full generality.

\section{Preliminaries}\label{prelims}

We present the notation and main tools  used in this paper and refer to the standard references 
\cite{Knus}, \cite{BOI}, \cite{Lam} and \cite{Sch} as well as to \cite{A-U-Kneb},  \cite{A-U-prime}
and \cite{A-U-PS} for the details.

Let $F$ be a field of characteristic different from $2$. We denote by $W(F)$ the Witt ring of $F$, by $X_F$  the space of orderings of $F$,
and by $F_P$ a real closure of $F$ at an ordering $P\in X_F$. 
We allow for the possibility that $F$ is not formally real, i.e. that $X_F=\varnothing$.
By an 
\emph{$F$-algebra with involution} we mean a pair $(A,\s)$ where $A$ is a 
finite-dimensional simple $F$-algebra with centre a field  $K=Z(A)$, equipped with an involution  $\s:A\to A$, such that $F = K \cap \Sym(A,\s)$, where $\Sym(A,\s):=\{a\in A \mid \s(a)=a\}$.  We also define $\Skew(A,\s):=\{a\in A \mid \s(a)=-a\}$.
Observe that $\dim_F K 
\le 2$.
We say that 
$\s$ is \df{of the first kind} if $K=F$ and \df{of the second kind} (or \emph{of unitary type}) otherwise. 
Involutions of the first kind can be further subdivided into those of \emph{orthogonal type} and those of \emph{symplectic 
type}, depending on the dimension of $\Sym(A,\s)$.
We let $\iota=\s|_{K}$ and note that $\iota =\id_F$ if $\s$ is of the first kind. 

The class of $F$-algebras with involution is often enlarged to include algebras with unitary involution with 
centre $F\x F$, cf.  \cite[2.B]{BOI} since it then becomes stable under scalar extension. We 
are not interested in such algebras  in 
this paper, since the objects we are interested in are trivial in this situation, cf. Remark~\ref{impromptu1}.

If $A$ is a division algebra, we call $(A,\s)$ an \df{$F$-division algebra with involution}.
We denote Brauer equivalence by~$\sim$ and isomorphism by $\cong$. Quadratic and hermitian forms are often just 
called \emph{forms}. The notation $\vf \leq \psi$ indicates that $\vf$ is a subform of $\psi$ and $\vf \simeq \psi$ indicates
that $\vf$ and $\psi$ are isometric.

Let $(A, \s)$ be an $F$-algebra with involution.
We denote by $W(A,\s)$ the Witt group  of Witt equivalence classes of nonsingular hermitian forms over $(A,\s)$, defined on 
finitely generated right $A$-modules. Note that $W(A,\s)$ is a $W(F)$-module.

For $a_1, \ldots, a_k \in F$ the notation $\qf{a_1,\ldots, a_k}$ stands for the quadratic form $(x_1,\ldots, x_k) \in F^k \mapsto  \sum_{i=1}^k a_i x_i^2 \in F$, as usual, whereas for $a_1, \ldots, a_k$ in $\Sym(A,\s)$ the notation $\qf{a_1,\ldots, a_k}_\s$ stands for the diagonal hermitian form  
\[ \bigl(  (x_1,\ldots, x_k), (y_1,\ldots, y_k)  \bigr) \in A^k \x A^k \mapsto   \sum_{i=1}^k \s(x_i) a_i y_i \in A.\]
In each case, we call $k$ the \emph{dimension} of the form. 

Let $h:M\x M\to A$ be a hermitian form over $(A,\s)$. We sometimes write $(M,h)$ instead of $h$.
The \emph{rank} of $h$, $\rk(h)$, is the rank of the $A$-module $M$. The set of elements represented by $h$ is denoted by
\[\das(h): =\{ u \in \Sym(A,\s) \mid  \exists x\in M\text{ such that } h(x,x)=u\}.\]
We say that $h$ is \emph{weakly isotropic} if there exists $\ell\in \N$ such that $\ell\x h$ is
isotropic, and  \emph{strongly anisotropic} otherwise. Similarly, we say that  $h$ is \emph{weakly 
hyperbolic} if there exists $\ell\in \N$ such that $\ell\x h$ is hyperbolic.

We denote by $\Int(u)$ the inner automorphism determined by $u \in A^\x$, where  $\Int(u)(x):= uxu^{-1}$ for $x\in A$.

It follows from the structure theory of $F$-algebras with involution  that $A$ is isomorphic to a full matrix algebra $M_\ell(D)$ for a unique 
$\ell\in \N$  (called the \emph{matrix size of $A$}) and an $F$-division algebra $D$
(unique up to isomorphism) that is equipped with an involution $\vt$ of the same kind as $\s$, cf. 
\cite[Theorem~3.1]{BOI}. 
For $B=(b_{ij})\in M_\ell(D)$ we let $\vt^t(B)=(\vt(b_{ji}))$.
By \cite[4.A]{BOI}, there exists $\ve \in \{-1,1\}$ and an invertible matrix $\Phi \in M_\ell(D)$ such that 
$\vt(\Phi)^t=\ve \Phi$ and 
$(A,\s)\cong (M_{\ell}(D), \ad_\Phi)$,
where $\ad_\Phi = \Int(\Phi) \circ \vt^t$. (In fact, $\Phi$ is the Gram matrix of an $\ve$-hermitian form over $(D,\vt)$.)
Note that $\Phi$ is only defined up to a factor in $F^\x$, with 
 $\ad_\Phi=\ad_{\lambda\Phi}$ for all $\lambda \in F^\x$ and that $\ve=1$ when $\s$ and $\vt$ are of the same type, cf. \cite[Theorem~4.2]{BOI}.  
We fix an isomorphism of $F$-algebras with involution
 $f:(A,\s)\to (M_{\ell}(D), \ad_\Phi)$.

Given an $F$-algebra with involution $(B,\tau)$ we denote by 
$\Herm_\ve(B,\tau)$ the category of (possibly singular) $\ve$-hermitian 
 forms over $(B, \tau)$, cf. \cite[p.~12]{Knus}. 
 We drop the subscript $\ve$ when $\ve=1$.
 The isomorphism $f$ trivially induces an equivalence of categories
 $f_*: \Herm(A,\s) \too \Herm(M_{\ell}(D), \ad_\Phi)$. Furthermore, the $F$-algebras with involution $(A,\s)$ and
 $(D,\vt)$ are Morita equivalent, cf. \cite[Chapter~I, Theorem~9.3.5]{Knus}. In this paper we make repeated use of a particular Morita 
 equivalence  between $(A,\s)$ and $(D,\vt)$, following the  approach in \cite{LU2} (see also \cite[\S2.4]{A-U-Kneb} 
 for the case of nonsingular forms and \cite[Proposition~3.4]{A-U-Kneb} for a justification of why
 using this equivalence is  as good as using any other equivalence for the purpose of computing signatures), namely:
\begin{equation}\label{diagram}
\xymatrix{
\Herm(A,\s)\ar[r]^--{f_*} &    \Herm(M_{\ell}(D),\ad_\Phi)\ar[r]^--{s}  &  \Herm_\ve(M_\ell(D), \vt^t) 
\ar[r]^--{g} & \Herm_\ve(D,\vt),}
\end{equation}
where  $s$ is the \emph{scaling by $\Phi^{-1}$} Morita equivalence, given by $(M,h)\mapsto (M, \Phi^{-1}h)$  
and
$g$ is the \emph{collapsing} Morita equivalence, given by $(M,h)\mapsto (D^k,b)$, 
where $k$ is the rank of $M$ as $M_\ell(D)$-module and $b$ is defined as follows: fixing an isomorphism $M \cong (D^\ell)^k$, $h$ can be identified with the form  $(M_{k,\ell}(D), \qf{B}_{\vt^t})$ for some matrix $B\in M_k(D)$ that satisfies $\vt^t(B)=\ve B$ and we take 
for $b$ the $\ve$-hermitian form whose Gram matrix is $B$. In particular the $\ve$-hermitian form  
$\qf{\diag(d_1,\ldots, d_\ell)}_{\vt^t}$ is mapped to $\qf{d_1,\ldots, d_\ell}_\vt$ by $g$.
Note that $\qf{B}_{\vt^t}(X,Y):= \vt(X)^t B Y$ for all $X,Y \in M_{k,\ell}(D)$.

Given an ordering $P\in X_F$ we defined a signature map
\[\sign_P^\eta: W(A,\s) \to \Z\]
via scalar extension to $F_P$ and Morita theory in  \cite{A-U-Kneb}. This map has many properties in common
with the usual Sylvester signature $\sign_P$ of quadratic forms (cf. \cite{A-U-Kneb}  and \cite[\S 2]{A-U-prime})
and reduces to $\pm \sign_P$ when $(A,\s)=(F,\id_F)$. See also \cite{A-U-PS} for a concise presentation as well as for the notation that we will use in this paper. In particular, recall that 
$\eta$ denotes a \emph{tuple of reference forms} for $(A,\s)$ and  that a Morita equivalence between $F$-algebras with involution
of the same type sends a tuple of reference forms to a tuple of reference forms, cf. the proof of 
\cite[Theorem~4.2]{A-U-prime}. 
Furthermore
\[\Nil[A,\s]:= \{P \in X_F \mid \sign^\eta_P=0\}\]
denotes the set of \emph{nil orderings} for $(A,\s)$, and depends only on the Brauer class of $A$ and the type of $\s$. Finally,
let
\[\wt X_F:=X_F \sm \Nil[A,\s],\]
which does not indicate the dependence on $(A,\s)$ in order to avoid cumbersome notation. 
By \cite[Corollary~6.5]{A-U-Kneb}, $\wt X_F$ is clopen in $X_F$ for the Harrison topology $\CT_H$,
and is therefore compact. We denote the restriction of the Harrison topology to $\wt X_F$ also
by $\CT_H$.

\begin{defi}
  Let $h \in \Herm(A,\s)$. We say that $h$ is \emph{universal} if $D_{(A,\s)} h =
  \Sym(A,\s)$.
\end{defi}

We denote the set of invertible elements in $\Sym(A,\s)$ by  $\Sym(A,\s)^\x$.

\begin{lemma}\label{morita-diag}
  Let $h$ be a hermitian form of rank $k$ over $(A,\s)$. There are $a_1, \ldots, a_k \in
  \Sym(A,\s)^\x \cup\{0\}$ such that $\ell \x h \simeq \qf{a_1,\ldots, a_k}_\s$.
  If $(A,\s) = (M_\ell(D),\vt^t)$ we can take $a_1, \ldots, a_k \in \Sym(D,\vt)
  \cdot I_\ell$.
\end{lemma}

\begin{proof} Since $f_*$ and $s$ preserve diagonal forms and isometries, we may assume  that $(A,\s) =
  (M_\ell(D),\vt^t)$. Let $d_1, \ldots,
  d_k \in \Sym(D,\vt)$ be such that $g(h) = \qf{d_1,\ldots, d_k}_\vt$. Then
  $g(\ell \x h) \simeq (\ell \x \qf{d_1}_\vt) \perp \cdots \perp (\ell \x
  \qf{d_k}_\vt)$. Therefore
  \begin{align*}
    \ell \x h &\simeq g^{-1}(\ell \x \qf{d_1}_\vt) \perp \cdots \perp g^{-1}(\ell \x
      \qf{d_k}_\vt)\\
      &\simeq \qf{\diag(d_1,\ldots,d_1)}_{\vt^t} \perp \cdots \perp
        \qf{\diag(d_k,\ldots,d_k)}_{\vt^t} \\
      &\simeq \qf{\diag(d_1,\ldots,d_1), \cdots,\diag(d_k,\ldots,d_k)}_{\vt^t}.\qedhere
  \end{align*}
\end{proof}

\begin{lemma}\label{isot-univ}
  Let $h$ be a nonsingular isotropic hermitian form over $(A,\s)$. Then $\ell \x h$ is
  universal.
\end{lemma}
\begin{proof}
By Morita theory $g(s(f_*(h)))$ is a nonsingular isotropic
 hermitian
  form over $(D,\vt)$, and thus there are $d_3, \ldots, d_k \in \Sym(D,\vt)$ such that
  $g(s(f_*(h))) \simeq \qf{-1,1,d_3,\ldots,d_k}_\vt$. Using the same method as in
  the proof of Lemma \ref{morita-diag} we obtain
  \[\ell \x s(f_*(h)) \simeq \qf{- I_\ell, I_\ell, d_3\cdot I_\ell ,\ldots,
  d_k\cdot I_\ell}_{\vt^t},\]
  and so
\[\ell \x h \simeq  \qf{- a,a}_\s \perp  h',\]  
for some  $a \in \Sym(A,\s)^\x$ and some  form $h'$ over $(A,\s)$.  Since $a$ is invertible, we have
$\qf{-a,a}_\s \simeq \qf{-1,1}_\s$ and a standard  argument shows that the hermitian form
  $\qf{-1,1}_\s$ over $(A,\s)$ is universal. The result follows.
\end{proof}

\begin{lemma}\label{isot-univ2} 
Let $a \in \Sym(A,\s) \sm \{0\}$. Then $\ell \x \qf{a}_\s$ represents an element in $\Sym(A,\s)^\x$
and $\ell \x \qf{a,-a}_\s$ is universal.
\end{lemma}

\begin{proof} By Lemma~\ref{morita-diag}, $\ell\x \qf{a}_\s \simeq \qf{a_1,\ldots, a_\ell}_\s$ for some $a_1,\ldots, a_\ell
\in \Sym(A,\s)^\x \cup\{0\}$. At least one $a_i$, say $a_1$, is nonzero since $\das (\ell\x \qf{a}_\s)$ contains $a\not=0$. Also, 
$\qf{a_1,-a_1}_\s \leq \ell\x \qf{a,-a}_\s$ and we conclude as in the proof of Lemma~\ref{isot-univ} since $a_1$ is invertible. 
\end{proof}

\begin{lemma}\label{diag-constant}
  There exist a nonsingular diagonal hermitian form $h_0$ over $(A,\s)$ and $k_0 \in \N$
  such that $\sign^\eta_P h_0 = k_0$ for every $P \in \wt X_F$.
\end{lemma}
\begin{proof}
  Let $\vf_0$ be a nonsingular diagonal hermitian form over $(A,\s)$ such that
  $\sign^\eta_P \vf_0 \not = 0$ for every $P \in \wt X_F$, cf. \cite[Proposition~3.2 and the remark following it]{A-U-prime}. Let $k_1, \ldots, k_s$ be the different values that $\sign^\eta_P
  \vf_0$ takes when $P$ varies in $\wt X_F$, and let $U_i$ be the clopen  set 
	 $\{P \in \wt X_F \mid \sign^\eta_P \vf_0 = k_i\}$ for
  $i=1, \ldots, s$. ($U_i$ is clopen since $\sign^\eta \vf_0$ is 
	continuous by \cite[Theorem~7.2]{A-U-Kneb} and $\wt X_F$ is clopen.)
	The result is obtained by induction on $s$, using the
  following fact. 

  Fact: There is a diagonal hermitian form $\vf'$ over $(A,\s)$ such that
  $\sign^\eta \vf'$ has constant non-zero values on $U_1 \cup U_2, U_3, \ldots,
  U_s$.\\
  Proof of the fact: Let $q$ be a quadratic form over $F$  and let $r \in \N$ be such that $\sign q$ is
  equal to $2^r k_2$ on $U_1$, to $2^rk_1$ on $U_2$ and to $2^r$ on $U_3 \cup
  \cdots \cup U_s$, cf. \cite[Chapter~VIII, Lemma~6.10]{Lam}. Then $q\ox\vf_0$ is a
  hermitian form over $(A,\s)$ and $\sign^\eta (q\ox\vf_0)$ is
  equal to $2^rk_2k_1$ on $U_1$, to $2^rk_1k_2$ on $U_2$ and to $2^rk_i$ on
  $U_i$ for $i=3, \ldots, s$.
\end{proof}

Let
$\CT_0$ be the coarsest topology on $\wt X_F$ that makes  $\sign^\eta h$
continuous, for all nonsingular hermitian forms $h$ over $(A,\s)$, and $\CT_1$
the coarsest topology on $\wt X_F$ that makes all $\sign^\eta \qf{a}_\s$
continuous, for $a \in \Sym(A,\s)^\x$.

\begin{prop}
  The topologies $\CT_H$, $\CT_0$ and $\CT_1$ on $\wt X_F$ are equal.
\end{prop}
\begin{proof}
  Clearly $\CT_1 \subseteq \CT_0$. 
  We also know that the maps $\sign^\eta h$ are all continuous from $\wt X_F$ with the topology 
	$\CT_H$ to
  $\Z$ by \cite[Theorem~7.2]{A-U-Kneb}, so $\CT_0 \subseteq \CT_H$. The result will follow if we check that
  $\CT_H \subseteq \CT_1$.
  
  Let $h_0 = \qf{a_1, \ldots, a_n}_\s$ be the
  hermitian form from Lemma~\ref{diag-constant}. Observe that since $h_0$ is
  nonsingular we have $a_1, \ldots, a_n \in \Sym(A,\s)^\x$.

  We consider the open set $H(u) \cap \wt X_F$ in $\CT_H$. Let $q$ be a nonsingular
  diagonal quadratic form over $F$ such that $\sign q$ is equal to $2^r$ on
  $H(u) \cap \wt X_F$ and $0$ on $\wt X_F \setminus H(u)$ (cf.
  \cite[Chapter~VIII, Lemma~6.10]{Lam} and since $\wt X_F$ is clopen in $X_F$). 
	The hermitian form $q \ox h$ is a
  diagonal form with invertible coefficients, so the map $\sign^\eta q \ox h$ is
  continuous on $\wt X_F$ equipped with $\CT_1$.
  
  By choice of $h_0$ and $q$,  $\sign^\eta q\ox h_0$ is equal to $2^rk_0$ on
  $H(u) \cap \wt X_F$ and to $0$ on $\wt X_F \setminus H(u)$.  Then $H(u) \cap
  \wt X_F = (\sign^\eta q\ox h_0)^{-1}(2^rk_0) \in \CT_1$.
\end{proof}

\section{Prepositive cones and positive cones}\label{sec:pos_cone}

The objective of this section is to introduce a notion of ordering on
  algebras with involution that corresponds to non-zero signatures of hermitian
  forms. The precise statement of this correspondence will appear later in the paper, in
  Corollary~\ref{bij}.

For the remainder of the paper we fix some   field $F$ of characteristic not $2$ and some $F$-algebra with involution $(A,\s)$. Let $(D,\vt)$, $\ad_\Phi$, $\ell$ and $\ve$ be as in Section~\ref{prelims}. 
We make the convention that orderings in $X_F$ always contain $0$.

\begin{defi}\label{def-preordering}
A \emph{prepositive cone $\CP$ on }$(A,\s)$ is a
  subset $\CP$ of $\Sym(A,\s)$  such that 
  \begin{enumerate}[(P1)]
    \item $\CP \not = \varnothing$;
    \item $\CP + \CP \subseteq \CP$; 
    \item $\s(a) \cdot \CP \cdot a \subseteq \CP$ for every $a \in A$;
    \item $\CP_F := \{u \in F \mid u\CP \subseteq \CP\}$ is an ordering on $F$.
    \item $\CP \cap -\CP = \{0\}$ (we say that $\CP$ is \emph{proper}).
  \end{enumerate}
  A prepositive cone $\CP$ is \emph{over $P\in X_F$} if $\CP_F=P$. 
  A \emph{positive cone} is a prepositive cone that is  maximal with respect to inclusion. 
  \end{defi}

Observe that $\{0\}$ is never a prepositive cone on $(A,\s)$ by (P4).

\begin{defi} \label{def-fr}  
We
  denote by $Y_{(A,\s)}$ the set of all prepositive cones on $(A,\s)$,  and by $X_{(A,\s)}$ the set of all  positive
  cones on $(A,\s)$.
  We say that $(A,\s)$ is \emph{formally real} if it has at least one prepositive
  cone.
\end{defi}

\begin{remark}\mbox{}
	\begin{enumerate}[(1)]
		\item We will show later in the paper that positive cones only exist over non-nil orderings, 
		and that for any such ordering there are exactly two positive cones over it,
		cf. Proposition~\ref{nn} and Theorem~\ref{positive=max}.

		\item There are prepositive cones that are not positive cones. A natural example can be
		constructed using the theory of Tignol-Wadsworth gauges, cf. \cite{T-W-2011}, as well as their
		links with positive cones, cf. \cite{A-U-gauges}, and will be presented in a forthcoming
		 paper.
	\end{enumerate}
\end{remark}

\begin{remark}\label{impromptu1} The reason we do not consider the case where  $Z(A)= F\x F$ in
this paper is as follows: in this case assume that $\CP$ is a prepositive cone on $(A,\s)$ and 
take $a\in \CP\sm\{0\}$.  Consider the hermitian form $\qf{a}_\s$. Then $\das(k\x\qf{a}_\s) \subseteq \CP$
for every $k\in \N$ by (P2) and (P3) and it follows from  \cite[Propositions~A.3 and B.8]{A-U-PS} that 
$\CP=\Sym(A,\s)$, contradicting (P5).
\end{remark}

The following proposition justifies why we use the terminology ``proper'' for (P5).

\begin{prop}\label{impromptu2}
  Let $\CP \subseteq \Sym(A,\s)$ satisfy properties \textup{(P1)} up to \textup{(P3)} and
  assume that $\CP$ does not satisfy \textup{(P5)}. Then $\CP = \Sym(A,\s)$.
\end{prop}

\begin{proof}
  By hypothesis there is $a \in \CP \setminus \{0\}$ such that $a, -a \in \CP$.
  Therefore $D_{(A,\s)} (k \x \qf{a,-a}_\s) \subseteq \CP$ for every $k \in \N$ by (P2) and (P3).
  Since $a\not=0$, Lemma~\ref{isot-univ2} tells us that
  $\ell \x \qf{-a,a}_\s$ is universal and the conclusion follows.
\end{proof}

\begin{lemma}\label{inv_pos}  
Let $\CP$ be a prepositive cone on $(A,\s)$. Then $\CP$ contains an invertible element.
\end{lemma}

\begin{proof} Consider the hermitian form $\qf{a}_\s$, where $a\in\CP\sm\{0\}$. By 
 Lemma~\ref{isot-univ2}, $\ell\x \qf{a}_\s$ represents an element $b \in \Sym(A,\s)^\x$. 
 Then $b\in\CP$ since by (P2) and (P3), 
 $\das(\ell\x\qf{a}_\s) \subseteq \CP$. 
\end{proof}

\begin{prop}\label{not_torsion} 
If $(A,\s)$ is formally real, then $W(A,\s)$ is not torsion.
\end{prop}

\begin{proof} 
Let $\CP$ be a prepositive cone on $(A,\s)$.
 Suppose $W(A,\s)$ is torsion and consider the hermitian form $\qf{a}_\s$, where $a\in\CP$ is invertible
 (cf. Lemma~\ref{inv_pos}). Then there 
exists $k\in \N$ such that $k\x \qf{a}_\s$ is hyperbolic. By Lemma~\ref{isot-univ} there exists $r\in\N$ such that 
$r\x \qf{a}_\s$ is universal. By (P2) and (P3), $\das(r\x\qf{a}_\s) \subseteq \CP$, contradicting (P5).
\end{proof}

The converse to the previous proposition also holds, as we will show in Proposition~\ref{tor_fr}.

\begin{cor}\label{cor:1} 
If $(A,\s)$ is formally real, then the involution $\vt$ on $D$ can be chosen
such that $\ve=1$.
\end{cor}

\begin{proof} If it is not possible to choose $\vt$ as indicated, then $(D,\vt,\ve)=(F,\id_F, -1)$ by 
\cite[Lemma~2.3]{A-U-PS}. Diagram~\eqref{diagram} then induces an isomorphism of Witt groups between $W(A,\s)$ 
and $W_{-1}(F,\id_F)$, the Witt
group of skew-symmetric bilinear 
forms over $F$, which is well-known to be zero, contradicting Proposition~\ref{not_torsion}.
\end{proof}

\noindent\textbf{Assumption for the remainder of the paper:} In view of this corollary we may and will assume  without loss of generality that the involution  $\vt$ on $D$  is chosen to be of the same type as $\s$  (i.e. such  that $\ve=1$) 
whenever $(A,\s)$ is formally real. 
\medskip

We collect some simple, but useful, properties of prepositive cones in the following proposition.

\begin{prop}\label{easy} 
Let $\CP$ be a prepositive cone on $(A,\s)$.
\begin{enumerate}[$(1)$]

\item Assume that $1\in \CP$. Then $\CP_F= \CP\cap F$. 
\item Let $\alpha \in \CP_F\setminus \{0\}$. Then $\alpha \CP =\CP$.
\end{enumerate}
\end{prop}

\begin{proof} (1) If $\alpha \in \CP_F$, then $\alpha = \alpha \cdot 1 \in \CP\cap F$. Conversely, assume that $\alpha\in
\CP\cap F$. If $\alpha \not\in \CP_F$, then $-\alpha \in \CP_F\sm\{0\}$. Therefore $-\alpha = -\alpha\cdot 1 \in \CP$, contradicting (P5) since $\alpha \in \CP$.

(2) We know that $\alpha \CP \subseteq \CP$. It follows from (P3) that $\alpha^{-1}\CP \alpha^{-1} \subseteq \CP$
and so $\alpha^{-1}\CP\subseteq \CP\alpha = \alpha\CP \subseteq \CP$. The result follows.
\end{proof}

Note that a characterization of the condition $1\in \CP$ is given in Corollary~\ref{1-pos}.

\begin{ex}
  Let $(A,\s)$  be an $F$-algebra with involution and let $P \in X_F$.  Let $h$
  be a hermitian form over $(A,\s)$ such that $\qf{\bar{u}} \ox h$ is
  anisotropic for every finite tuple $\bar{u}$ of elements of $P$. Then
  \[\CP:= \bigcup_{\bar{u} \in P^k,\ k\in \N} \das (\qf{\bar{u}} \ox h)\]
  is a prepositive cone on  $(A,\s)$.
  
 Also, if $\CP$ is a prepositive cone on $(A,\s)$ over $P\in X_F$, then for any diagonal hermitian form $h$ over $(A,\s)$
  with coefficients in $\CP$  we have
  \[\das (\qf{\bar{u}}\ox h)  \subseteq \CP  \]
  for every finite tuple $\bar{u}$ of elements of $P$. In particular the hyperbolic plane is never a subform of $\qf{\bar{u}}\ox h$. If  $h$ is in addition nonsingular, then  $h$ is strongly anisotropic  by Lemma~\ref{isot-univ}. 
\end{ex}

Most of the motivation behind the definition of prepositive cones comes from Examples~\ref{ex-psd} and \ref{mp} below.

\begin{ex}\label{ex-psd} 
Let $(A,\s) = (M_n(F),t)$, where $t$ denotes the transpose
  involution.   Let $P\in X_F$ and let $\CP$
  consist of those symmetric matrices in $M_n(F)$ that are  positive semidefinite
  with respect to $P$.  Then $\CP$ is a prepositive cone on $(M_n(F),t)$ over $P$.
  Note that if $M\in \CP$, then there exists $T \in \GL_n(F)$ such that $T^t MT$
  is  a diagonal matrix whose elements all belong to $P$. In particular, if $M$
  is invertible, $M$ is positive definite, and therefore $\sign_P^{\qf{I_n}_t} \qf{M}_t = n$. (Note that
  $\qf{I_n}_t$ is a reference form for $(M_n(F),t)$). We will observe later in Example~\ref{psd-nsd} that 
  $\CP$ and $-\CP$ are the only prepositive cones on $(M_n(F), t)$ over $P$.
\end{ex}

\begin{defi}\label{mp0} 
Let $(A,\s)$ be an $F$-algebra with involution and 
let $P \in X_F$. We define 
\[m_P(A,\s) := \max\{ \sign_P^\eta \qf{a}_\s \mid a\in \Sym(A,\s)^\x \}\]
(note that $m_P(A,\s)$ does not depend on the choice of $\eta$) and
\[\CM_P^\eta(A,\s):=\{ a \in \Sym(A,\s)^\x \mid  \sign^\eta_P \qf{a}_\s = m_P(A,\s)\}
  \cup \{0\}.\]
\end{defi}

\begin{ex}\label{mp} 
For every $F$-division algebra with involution $(D,\vt)$,  every tuple
  of reference forms $\eta$ for $(D,\vt)$ and every $P\in X_F\setminus \Nil[D,\vt]$, the set $\CM_P^\eta(D,\vt)$ is a
  prepositive cone on $(D,\vt)$ over $P$. Properties (P1) and (P4) are clear (for (P4) note that $\sign_P^\eta (q\ox h)=
  \sign_P q \cdot \sign_P^\eta h$, cf. \cite[Theorem~2.6]{A-U-prime}). We
  justify the others. 

  (P2): Let $a,b \in \CM_P^\eta(D,\vt) \setminus \{0\}$. Then $a+b$  is represented
  by $\qf{a,b}_\vt$.  Observe that since $P \in X_F\setminus \Nil[D,\vt]$, $m_P(D,\vt) > 0$ and thus $a
  \not = -b$ (because $a=-b$ would imply $m_P(D,\vt) = \sign^\eta_P \qf{a}_\s = -\sign^\eta_P
  \qf{b}_\s = -m_P(D,\vt)$, impossible).  Since $D$ is a  division algebra, $a + b$ is
  invertible and $\qf{a,b}_\vt \simeq \qf{a+b,c}_\vt$ for some $c \in
  \Sym(D,\vt)^\x$. Comparing signatures and using the fact that  $a, b \in
  \CM_P^\eta(D,\vt)$ we obtain $a+b \in \CM_P^\eta(D,\vt)$.

  (P3) follows from the fact that for any $a \in D^\x$ and any $b\in
  \Sym(D,\vt)^\x$ the forms $\qf{b}_\vt$ and $\qf{ \vt(a) b  a}_\vt$ are
  isometric. 

  (P5) follows from the fact that $P\not\in \Nil[D,\vt]$ and thus $m_P(D,\vt)>0$.
  
  We will see later, in Proposition~\ref{description}, that $\CM_P^\eta(D,\vt)$ is in fact a positive cone on $(D,\vt)$ over $P$.
  \end{ex}

\begin{remark} Observe that $\CP$ is a prepositive cone on $(A,\s)$ if and only if
  $-\CP$ is a prepositive cone on $(A,\s)$. This property of prepositive cones corresponds to the fact that
  if $\sign_P^\eta: W(A,\s) \to \Z$ is a signature, then so is $-\sign_P^\eta = \sign_P^{-\eta}$.  
    We will return to this observation
  later, cf. Proposition~\ref{pi-homeom}. 
\end{remark}

\begin{remark} The reason for requiring $\CP_F$ to be an ordering on $F$ in
  (P4) instead of just a preordering is as follows: in case $(A,\s)=
  (M_n(F),t)$, we want the prepositive cone $\CP$  on $(M_n(F),t)$ to be a subset
  of the set of positive definite (or negative definite) matrices (see
  Proposition \ref{split-positive} where we prove this result). 
   Having this
  requirement while  only asking that $\CP_F$ is a preordering forces $F$
  to be a SAP field   as we explain
	in the remainder of this remark.
	
	We first recall that a field $F$ has the Strong Approximation Property (SAP) if for any two
	disjoint closed subsets (in the Harrison topology) $A$ and $B$ of $X_F$, 
	there exists $a\in F$ such
	that $A\subseteq H(a)$ and $B\subseteq H(-a)$, cf. \cite[p.~66]{Prestel84}. This is equivalent to
	$F$ satisfying the Weak Hasse Principle, which asserts that a quadratic form over $F$ is weakly
	isotropic whenever it is indefinite with respect to all orderings of $F$, cf.
	\cite[Section~9]{Prestel84}.

  We thus assume for the remainder of this remark that axiom (P4) only requires
  $\CP_F$ to be a preordering. Let $q=\la a_1, \ldots, a_n\ra$ be a quadratic
  form over $F$ that is strongly anisotropic. Let $M=\diag(a_1, \ldots, a_n) \in M_n(F)$. 
  Then $M \in \Sym(M_n(F),t)$.  Let
  $\CP$ be the smallest subset of $M_n(F)$ containing $M$, that is closed under
  (P2) and (P3). In other words,  $\CP$ consists of all the elements in $M_n(F)$
  that are weakly represented by  the hermitian form $\qf{M}_t$. Since
  $\qf{M}_t$ is Morita equivalent to $q$ (via the map $g$),  $\qf{M}_t$ is strongly anisotropic.
  It follows that $\CP$ is a prepositive cone on $(M_n(F),t)$ with 
  $\CP_F$ a preordering on $F$.
  If $\CP$ is to only contain positive definite (or negative definite) matrices
  with respect to some $P\in X_F$,
  all the $a_i$ must have the same sign at $P$  and thus $q$
  must be definite at $P$. This shows that $F$ must satisfy the Weak
  Hasse Principle, i.e., that  $F$ is SAP, as recalled above.
\end{remark}

\begin{lemma}\label{same-base-ordering}
  Let $\CP \subseteq \CQ$ be prepositive cones in $(A,\s)$. Then $\CP_F = \CQ_F$.
\end{lemma}

\begin{proof}
  Let $\alpha \in \CQ_F$. Assume that $\alpha \not \in \CP_F$. Then $-\alpha \in
  \CP_F$ and thus $\alpha \CP = -\CP$, cf.  Proposition~\ref{easy}. Using that
  $\CP \subseteq \CQ$ we also obtain $\alpha \CP \subseteq \alpha \CQ \subseteq
  \CQ$. So $\CP, -\CP \subseteq \CQ$, contradicting that $\CQ$ is proper since
  $\CP \not = \{0\}$ (as observed after Definition~\ref{def-fr}). Therefore, $\CQ_F \subseteq \CP_F$ and
  the equality follows since they are both orderings on $F$. 
\end{proof}

\begin{defi}\label{extension}  
Let $(A,\s)$ be an $F$-algebra with involution and 
let $P \in X_F$. For a subset $S$ of $\Sym(A,\s)$, we define
  \[\CC_P(S) := \Bigl\{\sum_{i=1}^k u_i \s(x_i)s_ix_i \,\Big|\, k \in \N, \ u_i \in P,\ x_i
  \in A,\ s_i \in S\Bigr\}.\]
For a prepositive cone $\CP$  on $(A,\s)$ over $P$ and $a \in \Sym(A,\s)$, we define
\begin{align*}
\CP[a] &:=  \Bigl\{p + \sum_{i=1}^k u_i\s(x_i)ax_i \,\Big|\, p \in \CP,\ k \in
      \N, \ u_i \in P, \ x_i \in A \Bigr\}\\  
      &\phantom{:} = \CP + \CC_P(\{a\}).
\end{align*}      
\end{defi}
It is easy to see that both $\CC_P(S)$ and $\CP[a]$ are prepositive cones on
$(A,\s)$ over $P$ if and only if they are proper.

Prepositive cones on $(A,\s)$ only give rise to partial orderings on $\Sym(A,\s)$. The following result gives an 
approximation, for a positive cone $\CP$, to the property $F=P\cup -P$ for any ordering $P\in X_F$:

\begin{lemma}\label{complement}
  Let $\CP$ be a positive cone on $(A,\s)$ over $P \in X_F$
  and $a \in \Sym(A,\s) \setminus \CP$. Then there are $k
  \in \N$, $u_1, \ldots, u_k \in P \setminus \{0\}$ and $x_1, \ldots, x_k \in A
  \smz$ such that
  \[\sum_{i=1}^k u_i \s(x_i)ax_i \in -\CP.\]
\end{lemma}

\begin{proof}
  Assume for the sake of contradiction that the conclusion of the lemma does not
  hold, and consider $\CP[a]$. 
  It is easy to check that $\CP[a]$ satisfies axioms (P1) to (P4), and obviously it
  properly contains $\CP$ since it contains $a$. We now check that $\CP[a]$ is
  proper, thus reaching a contradiction. If $\CP[a]$ were not proper, we would have
  \[p + \sum_{i=1}^k u_i \s(x_i)ax_i = -(q + \sum_{j=1}^r v_j
  \s(y_j)ay_j),\]
  for some $k,r \in \N \cup \{0\}, u_i,v_j \in P\sm\{0\}$ and $x_i,y_j \in A \smz$ (and where neither side is $0$). Therefore
  \[\sum_{i=1}^k u_i \s(x_i)ax_i + \sum_{j=1}^r v_i \s(y_j)ay_j \in
  -\CP,\]
  a contradiction.
\end{proof}

\begin{remark}
Several notions of orderings have previously been considered in the special case
of division algebras with involution, most notably Baer orderings (see the
survey \cite{Craven-1995}).  The essential difference between our notion of (pre-)positive cone
and (for instance) Baer orderings, is that our definition is designed to correspond to
a pre-existing algebraic object: the notion of signature of hermitian forms (see
in particular axiom (P4) which reflects  the fact that the signature
is a morphism of modules, cf. \cite[Theorem~2.6(iii)]{A-U-prime}, and Example~\ref{mp}).

In order to achieve this, it was necessary to accept that the ordering defined
by a positive cone on the set of symmetric elements is in general only a partial ordering. It
gives positive cones a behaviour similar to the set of positive semidefinite
matrices in the algebra with involution $(M_n(F),t)$ (which provides, as seen above,
one of the main examples of (pre-)positive cones, cf. Example~\ref{ex-psd}).
\end{remark}

\section{Prepositive cones under Morita equivalence}\label{scaling}

For a subset $S$ of $\Sym(A,\s)$, we denote by $\Diag(S)$ the set of all
diagonal hermitian forms with coefficients in $S$.
We also recall: 

\begin{prop}[{\cite[Proposition~A.3]{A-U-PS}}]\label{ns-again}
Let $h$ be a hermitian form over
$(A,\s)$. Then
\[h\simeq h^\ns \perp 0, \]
where $h^\ns$ is a nonsingular hermitian form \tu{(}uniquely determined up to isometry\tu{)} and $0$ is the
zero form of suitable
rank. 
\end{prop}

\begin{thm}\label{morco}
  Let $(A,\s)$ and $(B,\tau)$ be two Morita equivalent 
  $F$-algebras with involution with $\s$ and $\tau$ of the same type, and fix   a Morita equivalence
  $\mor : \Herm(A,\s) \rightarrow \Herm(B,\tau)$.
  The  map
  \[\mor_*: Y_{(A,\s)} \rightarrow Y_{(B,\tau)},\]
  defined by
  \[\mor_*(\CP) := \bigcup\{D_{(B,\tau)} (\mor(h)) \mid h \in \Diag(\CP)\}\]
  is an inclusion-preserving bijection from $Y_{(A,\s)}$ to $Y_{(B,\tau)}$ that thus
  restricts to a bijection from $X_{(A,\s)}$ to $X_{(B,\tau)}$. Furthermore, if $\CP$ is over $P\in X_F$, then $\mor_*(\CP)$
  is also over $P$. 
\end{thm}

\begin{proof} We first show that $\mor_*(\CP)$ is a prepositive cone over $(B,\tau)$.
  By definition, $\mor_*(\CP)$ satisfies properties (P1), (P2) and (P3).
  We now show (P5): Assume there is $b \in \mor_*(\CP) \cap -\mor_*(\CP)$ such that $b\not=0$. Then there are diagonal hermitian forms
  $h_1$ and $h_2$   in $\Diag(\CP)$
  such that $b \in D_{(B,\tau)} (\mor(h_1))$ and $-b \in
  D_{(B,\tau)} (\mor(h_2))$. 
 By Lemma~\ref{isot-univ2} there is $k \in \N$ and $b_1 \in 
\Sym(B,\tau)^\x$ such that
$b_1 \in D_{(B,\tau)} ( \mor(k \x h_1))$ and $-b_1 \in D_{(B,\tau)}( \mor (k \x 
h_2))$. Since
$b_1$ is invertible we deduce (as for quadratic forms) that 
$\qf{b_1}_\tau \le \mor (k \x h_1)$
and $\qf{-b_1}_\tau \le \mor(k \x h_2)$. Therefore $\qf{b_1,-b_1}_\tau \le 
\mor (k \x (h_1 \perp h_2))$
and $\qf{b_1,-b_1}_\tau$ is nonsingular and clearly isotropic. So
$\mor^{-1}(\qf{b_1,-b_1}_\tau) \le k \x (h_1 \perp h_2)$ and 
$\mor^{-1}(\qf{b_1,-b_1}_\tau)$
is also nonsingular and isotropic by Morita theory. By Lemma~\ref{isot-univ} there is $k' \in 
\N$ such that
$kk' \x (h_1 \perp h_2)$ is universal, a contradiction to (P5) since
$\das (kk' \x (h_1 \perp h_2)) \subseteq \CP$.

  We prove
  (P4) by showing that $(\mor_*(\CP))_F = \CP_F$. Let $u \in \CP_F$ and let $b \in
  \mor_*(\CP)$. Then there is a diagonal hermitian form $h$ over $(A,\s)$ with coefficients in $\CP$ such that
  $b \in D_{(B,\tau)} (\mor(h))$. Since the map $W(A,\s) \to W(B,\tau)$ induced by $\mor$ is a morphism of $W(F)$-modules, the forms $\qf{u} \ox \mor(h)$ and $\mor(\qf{u}\ox h)$ are isometric and thus
 $ub \in D_{(B,\tau)} (\qf{u} \ox \mor(h)) =  D_{(B,\tau)} (\mor(\qf{u}\ox  h))$.
  This shows that $ub \in \mor_*(\CP)$ since $ \qf{u}\ox h \in \Diag(\CP)$
   and thus
  $\CP_F \subseteq (\mor_*(\CP))_F$. 
  
  Assume now that $\CP_F \subsetneqq (\mor_*(\CP))_F$, so
  that there is $u \in -\CP_F \setminus \{0\}$ such that $u \in (\mor_*(\CP))_F$.
  Let $b \in \mor_*(\CP) \setminus \{0\}$. Then by choice of $u$, $ub \in
  \mor_*(\CP)$, and since $-u \in \CP_F \subseteq (\mor_*(\CP))_F$ we also obtain $-ub
  \in \mor_*(\CP)$. Since $u \not =0$ we have $ub \not = 0$ and so obtain a contradiction to
  property (P5).

Finally, we show that $\mor_*$ is a bijection by showing that
 $(\mor^{-1})_*$ is a left inverse of $\mor_*$. The invertibility of
  $\mor_*$ then follows by swapping $\mor$ and $\mor^{-1}$. Let $\CP \in
  Y_{(A,\s)}$ and let $c \in (\mor^{-1})_*(\mor_*(\CP))$. Then there are elements $b_1, \ldots,
  b_r \in \mor_*(\CP)$ such that $c \in D_{(A,\s)} \mor^{-1}(\qf{b_1,\ldots,
  b_r}_\tau) = D_{(A,\s)} \mor^{-1}(\qnd{b_1,\ldots,   b_r}_\tau)$, and there
  are $a_1, \ldots, a_s \in \CP$ such that $b_1, \ldots, b_r \in
  D_{(B,\tau)}\mor(\qf{a_1,\ldots, a_s}_\s)$.
  Observe that by \cite[Proposition~2.10]{A-U-PS} each form $\qnd{b_i}_\tau$ is a subform of sufficiently many
  copies of $\mor(\qf{a_1,\ldots, a_s}_\s)$.  
  Hence there exists $r' \in \N$ such that
 $\qnd{b_1,\ldots,
  b_r}_\tau \le r' \x \mor(\qf{a_1,\ldots, a_s}_\s)$.   It follows that
  \begin{align*}
\mor^{-1}(\qnd{b_1,\ldots,
  b_r}_\tau) &= 
 \bigl(\mor^{-1}(\qf{b_1,\ldots,
  b_r}_\tau)\bigr)^\ns\\ 
  &\le \mor^{-1}[r' \x \mor(\qf{a_1,\ldots, a_s}_\s)] = r' \x
  \qf{a_1,\ldots, a_s}_\s.
  \end{align*} 
  Therefore, and since $D_{(A,\s)} \mor^{-1}(\qnd{b_1,\ldots,
  b_r}_\tau) = D_{(A,\s)} \mor^{-1}(\qf{b_1,\ldots,   b_r}_\tau)$, we obtain  
   $c \in D_{(A,\s)} (r' \x
  \qf{a_1,\ldots, a_s}_\s) \subseteq \CP$.
  
  Let now $a \in \CP$ and let $h = \mor(\qf{a}_\s)$. By Lemma~\ref{morita-diag}
  there are $t,\ell' \in \N$ and $b_1, \ldots, b_t \in D_{(B,\tau)} (\ell' \x h)$ such
  that $\ell' \x
  h \simeq \qf{b_1,\ldots,b_t}_\tau$. In particular $b_1, \ldots, b_t \in \mor_*(\CP)$
  and $h \le \qf{b_1,\ldots, b_t}_\tau$. Then $\qf{a}_\s = \mor^{-1}(h) \le
  \mor^{-1}(\qf{b_1,\ldots, b_t}_\tau)$, which implies $a \in (\mor^{-1})_*(\mor_*(\CP))$. 
Therefore $\CP = (\mor^{-1})_*(\mor_*(\CP))$.
\end{proof}

We can refine the description of the map $\mor_*$:
\begin{prop}\label{morco2}
  Let $\CP \in
  Y_{(A,\s)}$. Then, with the same hypotheses and notation as in Theorem~\ref{morco},
  \[\mor_*(\CP) := \bigcup\{D_{(B,\tau)} \mor(h) \mid h \in \Diag(\CP\cap A^\x)\}.\]
\end{prop}
\begin{proof}
  The inclusion from right to left is obvious. Let $b \in \mor_*(\CP)$. Then
  there exist $a_1, \ldots, a_r \in \CP$ such that $b \in D_{(B,\tau)}
  \mor(\qf{a_1,\ldots, a_r}_\s)$. By Lemma~\ref{morita-diag} there are $c_1,
  \ldots, c_{s} \in \Sym(A,\s)^\x$ such that $\ell \x \qf{a_1,\ldots,a_r}_\s
  \simeq \qf{c_1,\ldots,c_{s}, 0,\ldots,0}_\s$. It follows from this isometry that
  $c_1, \ldots, c_{s} \in \CP$, and thus 
  \[b \in D_{(B,\tau)}
  \mor(\qf{c_1,\ldots,c_{s},0,\ldots,0}_\s)= D_{(B,\tau)}
  \mor(\qf{c_1,\ldots,c_{s}}_\s),\] 
  since $\mor$ preserves forms with constant value zero.
 This proves the other inclusion.
\end{proof}

In addition to the general result, Theorem~\ref{morco}, we now give explicit descriptions of transferring  prepositive cones between $(A,\s) \cong (M_\ell(D), \Int(\Phi)\circ \vt^t)$ and  $(M_\ell(D),\vt^t)$ (scaling) and between $(M_\ell(D),\vt^t)$
and  $(D,\vt)$ (going up and going down).

\begin{prop}\label{prop-scaling}
  Let $\CP$ be a prepositive cone on $(A,\s)$ over $P$. Let $a \in \Sym(A,\s)^\x$. Then
  $a\CP$ is a prepositive cone on $(A, \Int(a) \circ \s)$ over $P$. This defines a
  natural inclusion-preserving bijection between $Y_{(A,\s)}$ and
  $Y_{(A,\Int(a) \circ \s)}$ and between $X_{(A,\s)}$ and $X_{(A,\Int(a) \circ \s)}$.
\end{prop}

\begin{proof}
  Properties (P1), (P2), (P4) and (P5) are clear. Let $b \in \CP$ and let $x \in A$. Then
  $(\Int(a) \circ \s)(x) ab x = a\s(x)a^{-1}abx = a\s(x)bx \in a\CP$, which proves
  (P3). The fact that the map $\CP \mapsto a\CP$  is a bijection, and preserves being proper as well as inclusions
   is clear.
\end{proof}

In the remainder of this section we describe the going up and going down correspondences, which are reminiscent of  the behaviour of positive semidefinite matrices.

\subsection{Going up}\label{up}
Let $\CP$ be a prepositive cone  on $(D,\vt)$ over $P \in X_F$.  We define
\[\PSD_\ell(\CP): = \{B \in \Sym(M_\ell(D), \vt^t) \mid \forall X \in D^\ell \quad \vt(X)^t B
X \in \CP\}.\]
The following result is straightforward, since any matrix in
$\Sym(M_\ell(D),\vt^t)$ is the matrix of a hermitian form over $(D,\vt)$ and thus can be diagonalized by congruences. 

\begin{lemma}\label{equiv-up}
  For $B \in \Sym(M_\ell(D), \vt^t)$ the following are equivalent:
  \begin{enumerate}[$(1)$]
    \item $B \in \PSD_\ell(\CP)$.
    \item There is $G \in \GL_\ell(D)$ such that $\vt(G)^t B G$ is diagonal with
      diagonal elements in $\CP$.
    \item For every $G \in \GL_\ell(D)$ such that $\vt(G)^t B G$ is diagonal,
      the diagonal elements are in $\CP$.
  \end{enumerate}
\end{lemma}

\begin{lemma}
  $\PSD_\ell(\CP)$ is the closure of $\CP \cdot I_\ell$ under sums and the operation
$Z \mapsto \vt(X)^t ZX$, for $X \in M_\ell(D)$.
\end{lemma}

\begin{proof}
  That $\PSD_\ell(\CP)$ is closed under these operations is clear. Let $B \in
  \PSD_\ell(\CP)$. By Lemma \ref{equiv-up} there is $G \in \GL_\ell(D)$ such that
  $\vt(G)^t B G = \diag(a_1,\ldots, a_\ell)$ with $a_1, \ldots, a_\ell \in \CP$.
  Observe that, with $C= \diag(0,\ldots,0,1,0,\ldots,0)$,
  \[\diag(0,\ldots,0,a_i,0,\ldots,0) =
    \vt(C)^t(a_i \cdot I_\ell)C,\]
  and the result follows using sums of matrices. 
\end{proof}

\begin{prop}\label{PSD} 
Let $\CP$ be a prepositive cone on $(D,\vt)$ over $P\in X_F$. 
  \begin{enumerate}[$(1)$]
    \item $\PSD_\ell(\CP)$ is a prepositive cone on $(M_\ell(D), \vt^t)$ over $P$.
    \item If $\CP$ is a positive cone, then so is $\PSD_\ell(\CP)$.
  \end{enumerate}
\end{prop}
\begin{proof} (1)  The properties   (P1) to (P4) are easily verified.  
For (P5), we assume that $\PSD_\ell(\CP)$ is not proper, i.e. there is $B
  \in \Sym(M_\ell(D),\vt^t) \setminus \{0\}$ such that $B \in \PSD_\ell(\CP) \cap
  -\PSD_\ell(\CP)$. Let $X \in D^\ell$ be such that $\vt(X)^tBX \not = 0$ 
  (such an  $X$ exists  since we may assume that $B$ is diagonal). Then $\vt(X)^t B X \in \CP \cap
  -\CP$, contradicting that $\CP$ is proper.

  For (2), assume that there is $B \in \Sym(M_\ell(D), \vt^t)$
  such that $\PSD_\ell(\CP) \subsetneqq \PSD_\ell(\CP)[B]$ and $\PSD_\ell(\CP)[B]$ is proper (cf. Definition~\ref{extension} for the notation). 
  Since $B \not \in \PSD_\ell(\CP)$ there is $X_0 \in D^\ell$ such that $b_0:=\vt(X_0)^t
  B X_0 \not \in \CP$. Let $Z \in M_\ell(D)$ be the matrix  whose columns are
  all $X_0$. Then $B_0 := \vt(Z)^t B Z \in \Sym(M_\ell(D),\vt^t)$
    is the matrix with $b_0$ everywhere. We have $B_0 \in \PSD_\ell(\CP)[B]$, so
  $\PSD_\ell(\CP)[B_0]$ is proper (since $\PSD_\ell(\CP)[B]$ is proper by assumption).
  Let $E_i := \diag(0,\ldots,0,1,0,\ldots,0)$ (with $1$ in position $i$). Then $B_1 := \sum_{i=1}^\ell \vt(E_i)^tB_0E_i = \diag(b_0,\ldots,b_0) \in \PSD_\ell(\CP)[B_0]$, so $\PSD(\CP)[B_1]$ is proper.

  We claim that $\CP[b_0]$ is proper,  contradicting that $\CP$ is maximal:  otherwise we would
  have $p + \sum_{i=1}^k u_i \vt(x_i)b_0x_i = -(q + \sum_{j=1}^r v_j
  \vt(y_j)b_0y_j) \not = 0$ for some $p,q\in \CP, k,r \in \N, u_i, v_j \in P$ and $x_i, y_j \in D$. Then if $X_i = \diag(x_i, \ldots, x_i)$ and $Y_j
  = \diag(y_j,\ldots, y_j)$, we have
  \[p I_n + \sum_{i=1}^k u_i \vt(X_i)B_1X_i = -(qI_n + \sum_{j=1}^r v_j
  \vt(Y_j)B_1Y_j) \not = 0,\]
  contradicting that $\PSD_\ell(\CP)[B_1]$ is proper.
\end{proof}

\subsection{Going down}\label{down}
Let $\CP$ be a prepositive cone over $P$ on $(M_\ell(D), \vt^t)$. Denoting the matrix trace by $\tr$, we define
\[\Tr_\ell(\CP) :=  \{\tr(B) \mid B \in \CP\}.\]

\begin{lemma}\label{equiv-down}
  For $d \in \Sym(D,\vt)^\x$ the following are equivalent:
  \begin{enumerate}[$(1)$]
    \item $d \in \Tr_\ell(\CP)$.
    \item $d \in \{\vt(X)^t B X \mid B \in \CP, \ X \in D^\ell\}$.
    \item $\diag(d,d_2,\ldots,d_\ell) \in \CP$ for some $d_2,\ldots, d_\ell \in \Sym(D,\vt)$.
     \item $\diag(d,0,\ldots,0) \in \CP$. 
  \end{enumerate}
\end{lemma}

\begin{proof} $(1)\Rightarrow (2)$:  Let $d=\tr(B)$ for some $B =(b_{ij}) \in \CP$   and  consider 
the matrix $E_i := \diag(0,\ldots,0,1,0,\ldots,0)$ (where  the element $1$ is in position $i$). 
Then $B_i := \vt(E_i)^t B E_i = \diag(0,\ldots,0,b_{ii},0,\ldots,0) \in \CP$ and so $\diag(b_{11},\ldots, b_{\ell\ell}) = B_1 + \cdots + B_\ell \in \CP$. Thus $d = \vt(X)^t (B_1 + \cdots + B_\ell) X$, where $X=(1 \cdots 1)^t \in D^\ell$.

$(2)\Rightarrow (3)$:
Let $d =   \vt(X)^tBX$ for some $B \in \CP$ and $X \in D^\ell$. The matrix $B$ is the
matrix of a hermitian form $h$ over $(D,\vt)$ and by hypothesis $d$ is
represented by $h$. Since $D$ is a division algebra, there is a diagonalization
of $h$ that has $d$ as its first entry, i.e.  there exist 
$G \in \GL_\ell(D)$ and $d_2,\ldots, d_\ell \in \Sym(D,\vt)$ such that $\vt(G)^tBG= \diag(d,d_2,\ldots,d_\ell)$.
  
$(3)\Rightarrow (4)$:  Follows from $ \diag(d,0,\ldots, 0)= \vt(X)^t  \diag(d,d_2,\ldots,d_\ell)  X $,   where
$X=\diag(1,0,\ldots, 0)$. 
 
$(4)\Rightarrow (1)$: Clear.  
\end{proof}

For future use we note the equality
\begin{equation}\label{trace} 
\Tr_\ell(\CP) = \{\vt(X)^t B X \mid B \in \CP, \ X \in D^\ell\},
\end{equation}
trivially given by $(1) \Leftrightarrow (2)$ in Lemma~\ref{equiv-down}.

\begin{prop}\label{tr} 
Let $\CP$ be a prepositive cone  on $(M_\ell(D), \vt^t)$ over $P$. 
  \begin{enumerate}[$(1)$]
    \item $\Tr_\ell(\CP)$ is a prepositive cone on $(D,\vt)$  over $P$.

    \item If $\CP$ is a positive cone, then so is $\Tr_\ell(\CP)$.
  \end{enumerate}
\end{prop}

\begin{proof}
(1)  Axioms (P1) and  (P2)  are straightforward, while (P3) follows from \eqref{trace}.
      We  check axiom (P5) and assume that $\Tr_\ell(\CP)$ is not proper, i.e. there
      is $a \in D \setminus \{0\}$ such that $a \in \Tr_\ell(\CP) \cap -\Tr_\ell(\CP)$. 
      By Lemma~\ref{equiv-down},  $\diag(a, 0, \ldots, 0)$ and $\diag(-a, 0, \ldots, 0)$ are in $\CP$,
  contradicting that $\CP$ is proper. Axiom (P4) now follows using (P5).
  
(2) Assume that $\Tr_\ell(\CP)$ is not maximal. Then there is $d \in
      \Sym(D,\vt)$ such that $d \not \in \Tr_\ell(\CP)$ and $\Tr_\ell(\CP) \subsetneqq
      \Tr_\ell(\CP)[d]$, where $\Tr_\ell(\CP)[d]$  is proper.
      Since $d \not \in \Tr_\ell(\CP)$, the matrix
      $C:=\diag(d,0,\ldots,0)$ is not in $\CP$  by Lemma~\ref{equiv-down}. We show that $\CP[C]$ is
      proper, which will contradict the hypothesis that $\CP$ is maximal,
      finishing the proof. 
      If $\CP[C]$ is not proper, then $\CP[C] \cap -\CP[C] \not = \{0\}$,
      so there are $B,B' \in \CP$, $k,r \in \N$, $u_i, v_j \in P$ and $X_i,
      Y_j \in M_\ell(D) \setminus\{0\}$ such that
      \[B+\sum_{i=1}^k u_i \vt(X_i)^tCX_i = -(B'+\sum_{j=1}^r v_j
      \vt(Y_j)^tCY_j) \not = 0.\]
      Up to multiplying all terms in the above equality on the left by $\vt(J)^t$
      and on the right by $J$ for some well-chosen invertible matrix $J$, we can
      assume that the matrix $B_0:=B+\sum_{i=1}^k u_i \vt(X_i)^tCX_i$ is diagonal.

      Let $k_0 \in \{1,\ldots,\ell\}$ be such that the $(k_0,k_0)$-th coordinate of
      $B_0$ is non-zero, and let $E \in D^\ell$ be the column vector with all coordinates $0$
      except for a $1$ at coordinate $k_0$. Then
      \[\vt(E)^tBE+\sum_{i=1}^k u_i \vt(X_iE)^tCX_iE =
      -(\vt(E)^tB'E+\sum_{j=1}^r v_j \vt(Y_j E)^t C Y_j E),\]
      where the left-hand side is non-zero by choice of $E$. Since both sides
      belong to $\Tr_\ell(\CP)[d]$ which is proper, we get a contradiction.
\end{proof}

\begin{prop}\label{correspondences}
  The maps 
  \[Y_{(D,\vt)} \rightarrow Y_{(M_\ell(D), \vt^t)},\ \CP \mapsto
  \PSD_\ell(\CP)\] 
  and 
  \[Y_{(M_\ell(D),\vt^t)} \rightarrow Y_{(D,\vt)},\  \CP \mapsto
  \Tr_\ell(\CP)\]
  are inverses of each other, and restrict to maps 
  \[X_{(D,\vt)}
  \rightarrow X_{(M_\ell(D), \vt^t)},\  \CP \mapsto \PSD_\ell(\CP)\] 
  and
  \[X_{(M_\ell(D),\vt^t)} \rightarrow X_{(D,\vt)},\ \CP \mapsto \Tr_\ell(\CP).\]
\end{prop}

\begin{proof}
By Propositions~\ref{PSD} and \ref{tr} we only need to show that these maps are
  inverses of each other. 
The equality $\Tr_\ell(\PSD_\ell (\CP))= \CP$  for $\CP \in Y_{(D,\vt)}$ is straightforward using \eqref{trace}. 

We now show that, given $\CP \in Y_{(M_\ell(D), \vt^t)}$, we have
$\PSD_\ell (\Tr_\ell(\CP))=\CP$.  Let $B\in \PSD_\ell(\Tr_\ell(\CP))$. Then by Lemma~\ref{equiv-up} there
exists $G\in \GL_\ell(D)$ such that $\vt(G)^t B G =\diag(d_1,\ldots, d_\ell) $ with $d_1,\ldots, d_\ell \in \Tr_\ell(\CP)$. 
By Lemma~\ref{equiv-down}, $\diag(d_i,0,\ldots,0) \in \CP$ for $i=1,\ldots, \ell$. Using (P3) with permutation
matrices, followed by (P2) we obtain that $ \diag(d_1,\ldots, d_\ell) \in \CP$ and therefore $B= \vt(G^{-1})^t 
 \diag(d_1,\ldots, d_\ell) G^{-1} \in \CP$.
 
 Conversely, let $B\in \CP$. Since $B \in \Sym(M_\ell(D), \vt^t)$ there exists $G\in \GL_\ell(D)$ such that
 $\vt(G)^t B G =\diag(d_1,\ldots, d_\ell) \in \CP$. It follows from \eqref{trace} that $d_1,\ldots, d_\ell$ are in
 $\Tr_\ell(\CP)$.  Therefore, by Lemma~\ref{equiv-up}, $\diag(d_1,\ldots, d_\ell)$ and thus $B$ are in  $\PSD_\ell(\Tr_\ell(\CP))$.
\end{proof}

\begin{ex}\label{psd-nsd} 
Let $P \in X_F$. The only two prepositive cones on $(F, \id)$ over $P$ are $P$ and $-P$. Therefore, 
by Proposition~\ref{correspondences},
the only two
prepositive cones on $(M_n(F), t)$ over $P$  
are the set of symmetric positive semidefinite   matrices and the set of symmetric negative semidefinite   matrices
with respect to $P$.
\end{ex}

\section{Prepositive cones under field extension}

\subsection{Basic results on convex cones over ordered fields}

In this section we fix some  $P\in X_F$ and write  $\ge_P$ for the order relation defined by $P$.
We recall some basic concepts from convex geometry.  

We consider the usual euclidean inner product on $F^n$,
\[\qf{x ;  y} := x_1y_1 + \cdots +x_ny_n,\]
and the topology $\CT_P$ on $F^n$ that comes from the order topology on $F$.
We say that a nonempty subset $C$ of $F^n$ is a \emph{cone over $P$} if it satisfies (C\ref{cone-one}) below and 
a \emph{convex cone over $P$} if it satisfies (C\ref{cone-one}) and (C\ref{cone-two}).
\begin{enumerate}[(C1)]
  \item\label{cone-one} For every $a \in C$ and $r \in P$, $ra \in C$;
  \item\label{cone-two} For every $a,b \in C$, $a+b \in C$.
\end{enumerate}
We say that a cone $C$ over $P$ is \emph{pointed} if $C \cap -C = \{0\}$, that it is
\emph{closed} if it is closed in the topology $\CT_P$, and that it is \emph{full-dimensional}
if $\Span(C) = F^n$.
We define the \emph{dual cone} of $C$ by
\[C^* := \{v \in F^n \mid \qf{v ;  C} \ge_P 0\}.\]
Observe that $C^*$ is always a closed convex cone over $P$.
We say that a convex cone $C$ over $P$  is \emph{finitely generated} if there are $a_1,\ldots,a_k
\in F^n$ such that
\[C = \{\lambda_1 a_1 + \cdots + \lambda_ka_k \mid \lambda_1,\ldots,\lambda_k
\in P\}.\]
We recall the following well-known result, cf.  \cite{Charnes-Cooper}, as well as some immediate consequences:

\begin{thm}[Farkas' Lemma]
  A finitely generated convex cone over $P$ is the intersection of a finite number of
  closed half-spaces \tu{(}i.e. is polyhedral\tu{)} with respect to $\geq_P$.
\end{thm}

\begin{cor}
  A finitely generated convex cone over $P$ is closed in $\CT_P$.
\end{cor}

\begin{cor}\label{double-dual}
  If $C$ is a finitely generated convex cone over $P$,  then $(C^{*})^{*} = C$.
\end{cor}

\begin{lemma}\label{fd-p}
  Let  $C$ be a cone over $P$. The following are equivalent:
  \begin{enumerate}[$(1)$]
    \item\label{fd-p1} $C$ is full-dimensional.
    \item\label{fd-p2} $C^*$ is pointed.
  \end{enumerate}
\end{lemma}
\begin{proof}
$\eqref{fd-p1} \Rightarrow \eqref{fd-p2}$: Assume that $C^*$ is not pointed.
Then there is $u \in F^n \setminus \{0\}$ such that $u, -u \in C^*$. By
definition of $C^*$ we have $\la u ;  x \ra \ge_P 0$ and $\la -u ;  x \ra \ge_P 0$ for
every $x \in C$, so $\la u ; x \ra = 0$ for every $x \in C$, i.e. $u \in C^\perp$.
But $C^\perp = \{0\}$ since $C$ is full-dimensional, contradiction.

$\eqref{fd-p2} \Rightarrow \eqref{fd-p1}$: Assume that $C$ is not
full-dimensional, so $F^n = \Span(C) \perp W$ for some non-zero subspace $W$.
Then $W \subseteq C^*$ and $-W \subseteq C^*$, contradicting that $C^*$ is
pointed.
\end{proof}

\subsection{Prepositive cones under field extensions}

Let $P\in X_F$ and let  $\dim_F A=m$.  We identify $A$ with $F^{m}$ as $F$-vector space, so
that we can use coordinates in $F$. Let $t\in\N$ and let
$b_1, \ldots, b_t : A \times A \rightarrow A$ be $F$-bilinear maps. We write $\bar b =
(b_1, \ldots, b_t)$ and 
define
\[C_{\bar b} := \Bigl\{\sum_{j=1}^{s} \sum_{i=1}^{t} a_{i,j}
    b_i(x_{i,j},x_{i,j}) \,\,\Big|\,\, s \in \N, \ a_{i,j} \in P,\ x_{i,j} \in
    A\Bigr\}.\]
In other words, $C_{\bar b} $ is the convex cone in $A$ over $P$ generated by the elements
$b_i(x,x)$ for $i = 1, \ldots, t$ and $x \in A$.

\begin{lemma}\label{bilin-closed}
  $C_{\bar b}$ is a finitely generated convex cone.
\end{lemma}

\begin{proof}
  It suffices to prove the statement for $t=1$, so for a single bilinear map $b$. Let
  $\{e_1,\ldots, e_m\}$ be a basis of $F^m$. Every element of $F^m$ can be
  written as $\sum_{i=1}^m \ve_i a_i e_i$ with $\ve_i \in \{-1,1\}$
  and $a_i \in P$. It follows that the convex cone generated by the elements $b(x,x)$
  for $x \in F^m$ is
  \begin{align*}
    \Bigl\{\sum_{k=1}^q u_k b\bigl(\sum_{i=1}^m & \ve_ia_ie_i,   \sum_{j=1}^m
    \delta_jc_je_j\bigr) \,\Big|\, q \in \N, u_k,a_i,c_j \in P, \ve_i, \delta_j \in
     \{-1,1\}\Bigr\} \\
    &= \Bigl\{\sum_{k=1}^q \sum_{i,j = 1}^{m} \ve_i\delta_j u_k a_ic_j
    b(e_i,e_j) \,\Big|\, q \in \N, u_k,a_i,c_j \in P, \ve_i, \delta_j \in
    \{-1,1\}\Bigr\} \\
    &= \Bigl\{\sum_{r=1}^s v_r b_r  \,\Big|\, s \in \N, v_r \in P, b_r \in \{\pm b(e_i,e_j)
   \mid i,j=1,\ldots,n\}\Bigr\},
  \end{align*}
  which is finitely generated.
\end{proof}

\begin{lemma}\label{closed-convex}
  Let $a_1, \ldots, a_s \in \Sym(A,\s)$. Then $\CC_P(a_1,\ldots, a_s)$ is a finitely generated
  convex cone over $P$.
\end{lemma}
\begin{proof}
  We have
  \[\CC_P(a_1,\ldots,a_s) = \Bigl\{\sum_{i=1}^k u_i \s(x_i)c_ix_i \,\,\Big|\,\, k \in \N,\
  u_i \in P,\ x_i \in A,\ c_i \in \{a_1,\ldots,a_s\}\Bigr\}.\]
  Define $b_i : A \x A \rightarrow A$ by $b_i(x,y) = \s(x)a_iy$. The map $b_i$ is
  $F$-bilinear and
  \begin{align*}
    \CC_P(a_1,\ldots,a_s) &= \Bigl\{\sum_{j=1}^r \sum_{i=1}^s u_{i,j}
      b_i(x_{i,j},x_{i,j}) \,\,\Big|\,\, r \in \N,\ u_{i,j} \in P,\ x_{i,j} \in A\Bigr\} \\
      &= C_{\bar b}.
  \end{align*}
  So by Lemma~\ref{bilin-closed}, the cone $\CC_P(a_1,\ldots,a_s) = C_{\bar
  b}$ is finitely generated.
\end{proof}

\begin{lemma}\label{closed-sigma}
  Let $(L,Q)$ be an ordered field extension of $(F,P)$. Let $\CF$ be a set
  of $F$-linear forms on $A$ such that $\bigcap_{f \in \CF} f^{-1}(P)$ is closed
  under $x \mapsto \s(y) x y$ for every $y \in A$. Let $a \in \bigcap_{f \in
  \CF} f^{-1}(P)$. Then
  \[(\s \ox \id)(z)\cdot (a \ox 1)\cdot z \in \bigcap_{f \in \CF} (f \ox \id)^{-1}(Q)\]
  for every $z \in A \ox_F L$.
\end{lemma}
\begin{proof}
  Let $z = \sum_{i=1}^k x_i \ox \alpha_i \in A \ox_F L$ and $f\in \CF$. Then
  \begin{equation}\label{lin-f}
    (f \ox \id)[(\s \ox \id)(z)\cdot (a \ox 1) \cdot z] = \sum_{i,j=1}^k f(\s(x_i)ax_j)
    \alpha_i\alpha_j.
  \end{equation}
  The map $q_f : A \rightarrow F$, $x \mapsto f(\s(x)ax)$ is a quadratic form over
  $F$, so there is an orthogonal basis $\{e_1,\ldots,e_{m}\}$ of $A$ for
  $q_f$. Therefore, writing $z = \sum_{i=1}^{m} (e_i \ox 1) \beta_i$  
   with
  $\beta_i \in L$, we obtain in \eqref{lin-f}:
  \begin{align*}
    (f \ox \id)[(\s \ox \id)(z)\cdot (a \ox 1) \cdot z] &= \sum_{i,j=1}^m f(\s(e_i)ae_j)
      \beta_i \beta_j \\
      &= \sum_{i=1}^{m} f(\s(e_i)ae_i) \beta_i^2.
  \end{align*}
  Since $f(\s(e_i)ae_i) \in P$ by choice of $a$, we obtain that $(f \ox \id)[(\s
  \ox \id)(z)(a \ox 1) z] \in Q$, proving the result.
\end{proof}

\begin{prop}\label{extension2}
  Let $(L,Q)$ be an ordered field extension of $(F,P)$ and let $\CP$ be a prepositive cone on
  $(A,\s)$ over $P$. Then $\CP \ox 1 := \{a \ox 1 \mid a \in \CP\}$ is
  contained in a prepositive cone on $(A \ox_F L, \s \ox \id)$ over $Q$.
\end{prop}
\begin{proof}
  It suffices to show that $\CC_Q(\CP \ox 1)$ is a prepositive cone on $(A \ox_F L,
  \s \ox \id)$ over $Q$, i.e. that it is proper. This is equivalent to showing that it is pointed as a convex cone, and for this it suffices to
  show that $\CC_Q(\{a_1 \ox 1, \ldots, a_r \ox 1\})$ is pointed, for every $a_1,
  \ldots, a_r \in \CP$ and $r\in \N$.

  Since $\CP$ is a prepositive cone over $P$, we know that $C:=\CC_P(\{a_1,\ldots, a_r\})$ is
  pointed, therefore by Corollary~\ref{double-dual} and Lemma~\ref{fd-p}, $C^*$ is full-dimensional. It follows
  that $C^*_Q$, the convex cone generated  by $\{b \ox 1 \mid b \in C^*\}$  over $Q$ is
  full-dimensional in $A \ox_F L$, and thus,  by Lemma~\ref{fd-p}, that its
  dual $(C^*_Q)^*$  is pointed.

  Claim: \begin{enumerate}[(1)]
    \item $(C^*_Q)^*$ contains $a_i \ox 1$ for $i=1,\ldots,r$;
    \item $(C^*_Q)^*$ contains $(\s \ox \id)(z)\cdot (a_i \ox 1) \cdot z$ for $i=1,\ldots,r$ and
      $z \in A \ox_F L$;
    \item $(C^*_Q)^*$ contains $\CC_Q(\{a_1 \ox 1, \ldots, a_r \ox 1\})$, the prepositive
      cone on $(A \ox_F L, \s \ox \id)$ over $Q$ generated by $a_1 \ox 1,
      \ldots, a_r \ox 1$.
  \end{enumerate}
  Proof of the claim: We first observe that
  \[(C^*_Q)^* = \bigcap_{b \in C^*} \{w \in A \ox_F L \mid \qf{b \ox 1 ; w} \in Q\}.\]
 
 (1) By definition of $C^*$ we have, for every $b \in C^*$, $\qf{b ; a_i} \in
      P$, i.e. $\alpha := \sum_{j=1}^{m} b_j \cdot (a_i)_j \in P$.  Therefore $\alpha
      \ox 1 = \sum_{j=1}^{m} (b_j \ox 1) ((a_i)_j \ox 1) \in Q$, i.e. $\qf{b
      \ox 1 ; a_i \ox 1} \in Q$, proving that $a_i \ox 1 \in (C^*_Q)^*$.
 
 (2)  By Lemma~\ref{closed-convex} we know that $C$ is a finitely generated convex cone
      over $P$, so by Corollary~\ref{double-dual}, $(C^{*})^{*} = C$, i.e.
      \[C = \bigcap_{b \in C^*} (f_b)^{-1}(P),\]
      where $f_b(x) := \qf{b;x}$. By Lemma~\ref{closed-sigma} it follows that
      \[\bigcap_{b \in C^*} (f_b \ox 1)^{-1}(Q)\]
      contains $(\s \ox \id)(z)\cdot (a_i \ox 1)\cdot z$ for every $i=1,\ldots,r$ and $z \in A
      \ox_F L$. The result follows since a direct verification shows that $(f_b
      \ox \id)(z) = \qf{b \ox 1; z}$ for every $z \in A \ox_F L$.

(3) By construction $(C^*_Q)^*$ is a convex cone over $Q$, so is closed under
      sum and multiplication by elements of $Q$. Using the second item in the
      claim, it follows at once that it contains $\CC_Q(\{a_1 \ox 1, \ldots,
      a_r \ox 1\})$.
     This finishes the proof of the Claim.

 It follows from (3) that $\CC_Q(\{a_1 \ox 1, \ldots, a_r \ox 1\})$  is pointed,
  which proves the result.
\end{proof}

\section{Existence of positive involutions}

The notion of positive involution  seems to go back to Albert, see for instance \cite{Albert55}, and Weil  \cite{Weil} 
and plays a central role in the paper \cite{P-S} by Procesi and Schacher.

Denote the reduced trace of $A$ by $\Trd_A$ and let $P\in X_F$.
Recall from \cite[Definition~1.1]{P-S}
that the involution $\s$ is called \emph{positive at $P$} whenever the form 
$A\x A\to K,\ (x,y)\mapsto \Trd_A(\s(x)y)$ is positive semidefinite at $P$ (hence positive definite at $P$ since it is 
nonsingular). For more details we refer to \cite[Section~4]{A-U-PS}.

In this section we will show that for any given $P \in \wt X_F$ there exists an involution on $A$ that is positive at $P$,
cf. Theorem~\ref{main_pos}.

Let $P\in X_F$.  If $\s$ is an involution of the first kind, $Z(A)=F$ and 
$A\ox_F F_P$ is isomorphic to a matrix algebra over $F_P$ or $(-1,-1)_{F_P}$. If $\s$ is of the second kind, $Z(A)$ 
is a quadratic extension of $F$  and $A\ox_F F_P$ is isomorphic to a matrix algebra over $F_P(\sqrt{-1})$, $F_P \x F_P$ 
or $(-1,-1)_{F_P} \x     (-1,-1)_{F_P}$, where the last two cases occur whenever $Z(A)\ox_F F_P\cong F_P\x F_P$, in which case $A\ox_F F_P$ is semisimple. Hence there exists a unique integer $n_P$ such that 
\begin{equation}\label{def:np}
A\ox_F F_P\cong M_{n_P}(D_P),
\end{equation}
where 
\begin{equation}\label{def:dp}
D_P \in \{F_P,  (-1,-1)_{F_P}, F_P(\sqrt{-1}),   F_P\x F_P, (-1,-1)_{F_P} \x     (-1,-1)_{F_P}\}.
\end{equation}
In light of this discussion we
define the following subsets of $X_F$:
\begin{equation}\label{def:subsets}
\begin{aligned}
    X_{\textup{rcf}} & := \{P \in X_F \mid D_P = F_P\} \\
    X_{\textup{quat}} &:= \{P \in X_F \mid D_P = (-1,-1)_{F_P}\} \\
    X_{\textup{acf}} &:= \{P \in X_F \mid D_P =  F_P(\sqrt{-1})\} \\
    X_{\textup{d-rcf}} &:= \{P \in X_F \mid D_P = F_P \x F_P\} \\
    X_{\textup{d-quat}}& := \{P \in X_F \mid D_P = (-1,-1)_{F_P} \x
    (-1,-1)_{F_P}\}.
  \end{aligned}
\end{equation}  
Note that the value of $n_P$ is constant on each of these sets, since it only depends on $\dim_{F_P} D_P$ by
\eqref{def:np}.

\begin{lemma}\label{type-of-splitting}
 Each of the five sets in \eqref{def:subsets} is a clopen subset of $X_F$.
\end{lemma}
\begin{proof}
  Since $X_F$ is the disjoint union of these five sets, it suffices to show that
  they are all open. We only do this for $X_{\textup{rcf}}$, the other cases are similar.
 Let $P \in X_{\textup{rcf}}$. Then
  there is a finite field extension $L$ of $F$ such that $L \subseteq F_P$ and
  $A \ox_F L  \cong M_{n_P}(L)$. Since an ordering $Q \in X_F$ extends to $L$ if and only
  if $\sign_Q ((\Tr_{L/F})_* \qf{1}) > 0$ (see \cite[Chapter~3, Theorem~4.4]{Sch}),
  the set $U := \{Q \in X_F \mid Q \text{ extends to } L\}$ is clopen in $X_F$
  and contains $P$.  Then for $Q \in U$ we have $L \subseteq F_Q$ and thus $A \ox_F
  F_Q \cong (A \ox_F L) \ox_L F_Q \cong M_{n_P}( L) \ox_L F_Q \cong M_{n_P}( F_Q)$, so $U \subseteq X_{\textup{rcf}}$.
\end{proof}

\begin{remark}\label{nil-and-non}
The algebra $D_P$ carries an involution $\vt_P$ of the same kind as $\s$, and $\s\ox\id_{F_P}$ is adjoint to an 
$\ve_P$-hermitian form over $(D_P,\vt_P)$ with $\ve_P \in \{-1,1\}$. Note that if $\s$ and $\vt_P$ are of the first kind, 
they have the same type if 
$\ve_P=1$.
Recall from \cite[Section~3.2]{A-U-Kneb}
that  $P\in \Nil[A,\s]$ if we can choose $\ve_P$ and $\vt_P$ such that $\ve_P=-1$ and 
\begin{equation*}
(D_P,\vt_P) \in \{(F_P, \id),  ((-1,-1)_{F_P}, \bbar),   (F_P\x F_P,\swap) , ((-1,-1)_{F_P} \x     (-1,-1)_{F_P}, \swap)\}
\end{equation*}
and that $P\in \wt X_F = X_F \sm \Nil[A,\s]$ if we can choose $\ve_P$ and $\vt_P$ such that $\ve_P=1$ and 
\begin{equation}\label{non-nilcases}
(D_P,\vt_P) \in \{ (F_P, \id),  ((-1,-1)_{F_P}, \bbar), (F_P(\sqrt{-1}), \bbar)\}.
\end{equation}
Here $\id$ denotes the (orthogonal) identity involution on $F_P$,
$\bbar$ denotes the (unitary) conjugation involution or the (symplectic) quaternion conjugation involution 
and $\swap$ denotes the (unitary) exchange involution.

Note that the list in  \cite[Section~3.2]{A-U-Kneb} is missing the double quaternion case, an omission corrected in
\cite[Section~2]{A-U-prime}.
\end{remark}

\subsection{Prepositive cones in the almost split case}\label{almost-split}

For convenience  we record the following trivial fact:
\begin{fact}  \label{subreal0}
Let $B = \diag(b_1, \ldots, b_n)$ be a diagonal matrix in $M_n(F)$.
\begin{enumerate}[(1)]
\item   Let $\pi
  \in S_n$ and denote the associated permutation matrix by $P_\pi$. Then $P_\pi B
  {P_\pi}^t = \diag(b_{\pi(1)}, \ldots, b_{\pi(n)})$.
  
\item    Let $i,j
  \in \{1, \ldots, n\}$. Then there is $E \in M_n(F)$ such that $E B E^t$ is the
  matrix with zeroes everywhere, except for $b_i$ at coordinates $(j,j)$.
\end{enumerate}
\end{fact}

\begin{lemma}\label{subreal2}
  Let $P \in X_F$ and let $S \in M_n(F)$ be a diagonal matrix with at least two nonzero entries of different sign
  with respect to $P$. Let $\bar{
  \varepsilon} \in \{-1,0,1\}^n$. Then there is a diagonal matrix $S' \in M_n(F)$ such
  that
  \begin{enumerate}[$(1)$]
    \item $S'$ is weakly represented by $\la S \ra_{t}$ in $(M_n(F),t)$;
    \item $\sign_P S'_{ii} = \varepsilon_i$ for $i=1, \ldots, n$.
  \end{enumerate}
\end{lemma}

\begin{proof}
  Without loss of generality (by Fact~\ref{subreal0}(1)) we may reorder the elements in $\bar
  \varepsilon$ such that  its non-zero entries  are $\varepsilon_1, \ldots, \varepsilon_k$.
  For $i=1, \ldots, k$, by Fact~\ref{subreal0}(2) there are matrices $E_1, \ldots,
  E_k$ such that $S_i:=E_i S E_i^t$ has zeroes everywhere, except
  for an element of sign $\varepsilon_i$ at coordinates $(i,i)$. Now take $S' = S_1
  + \cdots + S_k$.
\end{proof}

The next result is a step towards the proof of Proposition~\ref{m_P=n_P} and
shows that our definition of prepositive cone
corresponds to positive semidefinite matrices (or negative semidefinite
matrices) in all  cases that are relevant.

\begin{prop}\label{split-positive}
  Let $P \in X_F$. Let $\CP$ be a prepositive cone  on $(M_\ell(D),
  \vt^t)$  over $P$, where $(D,\vt)$ is one of $(F, \id)$, $(F(\sqrt{-d}), \bbar)$,
  $((-a,-b)_F, \bbar)$  with $a,b,d \in P \setminus \{0\}$. Then there is
  $\varepsilon \in \{-1,1\}$ such that
  \begin{align*}
  \CP \subseteq \{C \in \Sym(M_\ell(D), \vt^t) \mid \exists G\in \GL_\ell(D)\text{ such that } \vt(G)^t  C  G \in \Diag(\ve P)\}.
 \end{align*}
\end{prop}

\begin{proof}
  Let $C \in \CP$. Since $C \in \Sym(M_\ell(D), \vt^t)$, it is the matrix of
  some hermitian form over $(D,\vt)$, so there is $G \in \GL_\ell(D)$ such
  that $\vt(G)^t  C  G$ is diagonal, and since $\CP$ is a
  prepositive cone, this new matrix is also in $\CP$. Therefore, we can assume that
  $C$ is diagonal, and thus has diagonal coefficients in $F$.  Suppose for the
  sake of contradiction that $C$ has two  non-zero diagonal elements of
  different signs with respect to $P$. By Lemma~\ref{subreal2}, there are
  $u,v \in P\sm\{0\}$ such that the matrices $\diag(u,0, \ldots, 0)$ and
  $\diag(-v,0,\ldots, 0)$ belong to $\CP$. Since $\CP$ is closed under products
  by elements of $P$, it follows that the matrices $\diag(1,0,\ldots,0)$ and
  $\diag(-1,0,\ldots,0)$ are both in $\CP$, contradicting that $\CP$ is proper.

  Assume now that we have two matrices $B$ and $C$ in $\CP$ such that 
   $\vt(G)^t  B  G\in \Diag(P)$ and $\vt(H)^t  C  H\in \Diag(-P)$
   for some $G, H \in \GL_\ell(D)$.
  By the properties of $\CP$ and Fact~\ref{subreal0}(1), we may assume that
  $B$ and $C$ are diagonal with non-zero first diagonal element. As above, it
  follows that there are $u,v \in P\sm\{0\}$ such that the matrices $\diag(u,0,\ldots,
  0)$ and $\diag(-v,0,\ldots,0)$ both belong to $\CP$, and thus that
  $\diag(1,0,\ldots,0)$ and  $\diag(-1,0,\ldots,0)$ belong to $\CP$, contradiction.
\end{proof}

\subsection{Prepositive cones and positive involutions}

\begin{prop}\label{nn} Let $P\in X_F$. The following statements are equivalent:
\begin{enumerate}[$(1)$]
\item There is a  prepositive cone on $(A,\s)$ over $P$;
\item $P\in \wt X_F$.
\end{enumerate}
\end{prop}

\begin{proof} 
$(1) \Rightarrow (2)$: Assume that $P\in\Nil[A,\s]$. Then  the unique ordering on $F_P$
is in $\Nil[A\ox_F F_P, \s\ox\id]$ by Remark~\ref{nil-and-non}. Thus $W(A\ox_F F_P, \s\ox\id)$
is torsion by Pfister's local-global principle \cite[Theorem~4.1]{LU1}. But $(A\ox_F F_P, \s\ox\id)$
is formally real by Proposition~\ref{extension2}, which  contradicts 
Proposition~\ref{not_torsion}. 

$(2) \Rightarrow (1)$: By \cite[Remark~2.7]{A-U-PS} we may assume that $\ve=1$, i.e. that $\s$ and $\vt$ are
of the same type. It follows that $\Nil[D,\vt]=\Nil[A,\s]$ and so there exists a prepositive cone on $(D,\vt)$ over $P$ by Example~\ref{mp}. By Theorem~\ref{morco} there exists a prepositive cone on $(A,\s)$ over $P$. 
\end{proof}

\begin{prop}\label{m_P=n_P}
Let $P \in X_F$ and let $\CP$ be a prepositive cone on $(A,\s)$ over  $P$. Then 
 there exists $\ve \in \{-1, 1\}$
such that
\[\CP \cap A^\times \subseteq \{a \in \Sym(A,\s)^\x \mid \sign^{\eta}_{P} \qf{a}_\s = \ve n_P\}.\]
In particular, $m_P(A,\s) = n_P$  for every $P \in \wt X_F$.
\end{prop}

\begin{proof}
Observe that $P\in \wt X_F$ by Proposition~\ref{nn}.
By Proposition~\ref{extension2} there exists a prepositive cone $\CQ$ on   $(A \ox_F F_P, \s \ox \id)$ such that 
$\CP\ox 1 \subseteq \CQ$. Recall from \eqref{def:np} that there is an isomorphism $f_P:A\ox_F F_P \to M_{n_P} (D_P)$. 
We now apply part of diagram~\eqref{diagram} to $(A\ox_F F_P, \s\ox\id)$, adjusting the diagram 
\emph{mutatis mutandis}:
\[
\xymatrix{
\Herm(A\ox_F F_P,\s\ox\id )\ar[r]^--{(f_P)_*} &    \Herm(M_{n_P}(D_P),\ad_{\Phi_P})\ar[r]^--{s_P}  &  \Herm_{\ve_P}
(M_{n_P}(D_P), \vt_P^t) 
} 
\]
Since $\CQ$ is a prepositive cone on $(A \ox_F F_P, \s \ox \id)$,  we may assume
that  $\ve_P=1$ by Corollary~\ref{cor:1},  i.e. that $\vt_P^t(\Phi_P)=\Phi_P$. 
 Thus the prepositive cone $\CQ$ is transported
 to the prepositive cone $\Phi_P^{-1} f_P(\CQ)$ by Proposition~\ref{prop-scaling}. Since $P\in\wt X_F$, $(D_P, \vt_P)$ is in the list \eqref{non-nilcases}. Thus, by Proposition~\ref{split-positive}, 
there exists $\ve' \in \{-1,1\}$ such that $\ve'\Phi_P^{-1} f_P(\CQ)$
only contains positive semidefinite matrices with respect to the ordering on $F_P$. In particular,
every invertible element in $\Phi_P^{-1} f_P(\CQ)$ will have Sylvester signature $\ve' n_P$. 

Let $a\in \CP$ be invertible (cf. Lemma~\ref{inv_pos}). Then there exists $\delta \in \{-1,1\}$, depending only on $\eta$,
such that
\[\sign_P^\eta \qf{a}_\s =\sign^{\eta \ox 1} \qf{a\ox 1}_{\s\ox \id} = \sign^{ s_P\circ (f_P)_*  (\eta)} \qf{\Phi_P^{-1}(f_P(a\ox 1))}_{\vt_P^t}=\delta\ve' n_P,\]
where the final equality follows from the fact that $\Phi_P^{-1}(f_P(a\ox 1))$ is an invertible element in $\Phi_P^{-1} f_P(\CQ)$.  

Finally, if $P \in \wt X_F$, there exists a prepositive cone on 
$(A,\s)$ over $P$ by Proposition~\ref{nn} and thus, by the argument above, there is at least one element with signature $n_P$.
\end{proof}

The following result solves the question of the existence of positive involutions at a given ordering, a problem that seems not to have been treated as yet despite the appearance of positive involutions in the literature.

\begin{thm}\label{main_pos} 
Let  $P \in X_F$. The following statements are equivalent:
\begin{enumerate}[$(1)$]
\item There is an involution  $\tau$ on $A$ which is positive at $P$ and of the same type as $\s$;
\item $P \in \wt X_F = X_F\setminus \Nil[A,\s]$.
\end{enumerate}
\end{thm}

(Recall that for $\tau$ of the same type as $\s$, $\Nil[A,\s]= \Nil[A,\tau]$.)

\begin{proof} Let  $\tau$ be an involution on $A$ of the same type as $\s$ and let $\eta$ be a tuple of reference forms for $(A,\tau)$. Thus $\tau=\Int(b)\circ \s$ for some $b\in \Sym(A,\s)^\x$ 
and 
\begin{equation}\label{pi}
\sign_P^\eta \qf{1}_\tau = \sign_P^{b^{-1}\eta} \qf{b^{-1}}_\s, 
\end{equation}
cf. \cite[Theorem~4.2]{A-U-prime}.

$(1)\Rightarrow (2)$: If $\tau$ is positive at $P$, then 
$\qf{b^{-1}}_\s$ has nonzero signature at $P$ by  \cite[Corollary~4.6]{A-U-PS} 
and thus $P\in \wt X_F$. 

$(2)\Rightarrow (1)$: If $P\in \wt X_F$, there exists $b\in \Sym(A,\s)^\x$ such that $\sign_P^{\eta'} \qf{b^{-1}}_\s =n_P$ by Proposition~\ref{m_P=n_P}. Therefore, $\tau= \Int(b)\circ \s$ is positive at $P$  by \eqref{pi} and 
\cite[Corollary~4.6]{A-U-PS}, and of the same type as $\s$, cf. \cite[Propositions~2.7 and 2.18]{BOI}.
\end{proof}

\section{Positive cones}

We remind the reader of the notation of diagram~\eqref{diagram}, which will be used throughout
this section.

\subsection{Description of positive cones}
In this subsection, we show that there are exactly two positive cones for a given algebra with involution over any given non-nil ordering of the base field, cf. Theorem~\ref{positive=max}.

The following result is a reformulation of Proposition~\ref{m_P=n_P} for $F$-division algebras 
$(D,\vt)$ with involution of any kind
and completely describes their positive cones.

\begin{prop}\label{description}
  Let $\CP$ be a prepositive cone on $(D,\vt)$ over $P \in X_F$. Then
  $\CP \subseteq \CM^\eta_P(D,\vt)$ or $\CP \subseteq -\CM^\eta_P(D,\vt)$ \tu{(}cf. Definition~\ref{mp0}\tu{)}.

  In particular $\CM^\eta_P(D,\vt)$ and $-\CM^\eta_P(D,\vt)$ are the only
  positive cones over $P$, i.e. $X_{(D,\vt)} =
  \{-\CM^\eta_P(D,\vt), \CM^\eta_P(D,\vt) \mid P \in  X_F\sm \Nil[D,\vt] \}$.
\end{prop}

\begin{lemma}\label{CPSD}
  $\CC_P(\CM^{g^{-1}(\eta)}_P(M_\ell(D),\vt^t)) = \PSD_\ell(\CM^\eta_P(D,\vt))$. 
\end{lemma}

\begin{proof}
  Let $B \in \PSD_\ell(\CM_P^\eta(D,\vt))$. Then by Lemma~\ref{equiv-up} there
  is $G \in \GL_\ell(D)$ such that $\vt(G)^tBG = \diag(a_1, \ldots, a_\ell)$
  with $a_1, \ldots, a_\ell \in \CM^\eta_P(D,\vt)$, and we may assume that
  \[\vt(G)^tBG = \diag(a_1, \ldots, a_r,0,\ldots,0)\]
  with $a_1, \ldots, a_r \in \CM^\eta_P(D,\vt) \sm \{0\}$. Using that $a_i I_\ell \in
  \CM^{g^{-1}(\eta)}_P(M_\ell(D),\vt^t)$, it is now easy to represent $\vt(G)^tBG$,
  and thus $B$, as an element of $\CC_P(\CM^{g^{-1}(\eta)}_P(M_\ell(D),\vt^t))$, proving
  that $\PSD_\ell(\CM^\eta_P(D,\vt)) \subseteq \CC_P(\CM^{g^{-1}(\eta)}_P(M_\ell(D),\vt^t))$.
The result follows since $\PSD_\ell(\CM^\eta_P(D,\vt))$ is a positive cone by Proposition~\ref{description}
and Proposition~\ref{PSD}.
\end{proof}

Using the notation from Proposition~\ref{ns-again}
 we recall \cite[Definition~3.1]{A-U-PS}: 

\begin{defi} Let $P\in X_F$ and let $\eta$ be a tuple of reference forms for $(A,\s)$. An element
 $u\in \Sym(A,\s)$ is called \emph{$\eta$-maximal at $P$} if for every nonsingular hermitian form $h$ of rank equal to 
 the rank of $\qnd{u}_\s$, we have $\sign_P^\eta \qnd{u}_\s \geq \sign_P^\eta h$, where $\qnd{u}_\s$ denotes the
 nonsingular part of the form $\qf{u}_\s$. 
 
 If $Y\subseteq X_F$, then we say that $u$ is \emph{$\eta$-maximal on $Y$} if $u$ is $\eta$-maximal at $P$ for all
 $P\in Y$. 
\end{defi}

Observe that 
if $u$ is invertible,  $\qnd{u}_\s =\qf{u}_\s$ and thus  $u$ is $\eta$-maximal at $P$
if and only if $\sign_P^\eta \qf{u}_\s = m_P(A,\s)$ if and only if $u \in \CM^\eta_P(A,\s)$.
If $u$ is not invertible, we have the following equivalence:

\begin{lemma}\label{bleuarg}
  Let $u \in \Sym(A,\s)$ and let $P \in \wt X_F$. Then $u \in
  \CC_P(\CM^\eta_P(A,\s))$ if and only if $u$ is $\eta$-maximal at $P$.
\end{lemma}

\begin{proof}
  Let $k = \rk \qnd{u}_\s$. We have the following sequence of equivalences:
  \begin{align*}
    u \in   \CC_P & (\CM^\eta_P(A,\s))  \\
    &\Leftrightarrow f(u) \in
    \CC_P(\CM^{f_*(\eta)}_P(M_\ell(D),\Int(\Phi) \circ \vartheta^t))\\
      & \Leftrightarrow \Phi^{-1} f(u) \in
        \CC_P(\CM^{s\circ f_*(\eta)}_P(M_\ell(D),\vartheta^t))\\
      & \Leftrightarrow \Phi^{-1} f(u) \in
        \PSD_\ell(\CM^{g\circ s\circ f_*(\eta)}_P(D,\vartheta))\qquad [\text{by Lemma~\ref{CPSD}}] \\
      & \Leftrightarrow \exists G \in \GL_\ell(D) \ \exists d_1,\ldots, d_k \in
        \CM^{g \circ s \circ f_*(\eta)}_P(D,\vartheta) \setminus \{0\} \\
      &\qquad \vartheta(G)^t \Phi^{-1} f(u)G = \diag(d_1,\ldots,k_k,0,
        \ldots,0) \qquad [\text{by Lemma~\ref{equiv-up}}]     \\
      & \Leftrightarrow \exists d_1, \ldots, d_k \in \CM^{g \circ s \circ f_*(\eta)}_P(D,\vartheta) \quad 
        g \circ s \circ f_*(\qf{u}_\s) \simeq \qf{d_1, \ldots, d_k}_\vartheta
        \perp 0\\
      & \Leftrightarrow   \sign_P^\eta \qnd{u}_\s \text{ is maximal  among all
        forms of rank } k\\
     & \Leftrightarrow  \text{$u$  is  $\eta$-maximal at  $P$.}\qedhere
  \end{align*}
\end{proof}

\begin{thm}\label{positive=max}
  Let $\eta$ be a tuple of reference forms for $(A,\s)$ and let $\CP$ be a
  prepositive cone on $(A,\s)$ over $P \in X_F$. Then 
  \[\CP \subseteq \CC_P(\CM^\eta_P(A,\s)) \text{ or } \CP \subseteq
  -\CC_P(\CM^\eta_P(A,\s)).\]
  In particular
  \[X_{(A,\s)} = \{-\CC_P(\CM^\eta_P(A,\s)), \CC_P(\CM^\eta_P(A,\s)) \mid P \in  \wt X_F  \}\]
  and
  for each $\CP \in X_{(A,\s)}$, there exists $\ve\in\{-1,1\}$ such that
  $\CP \cap A^\x = \ve \CM^\eta_P(A,\s) \setminus \{0\}$.
\end{thm}
\begin{proof}
We only prove the second part of the theorem, since the first part follows from
it. We start with the description of $X_{(A,\s)}$. Let $\CP$ be a  positive cone on $(A,\s)$ over $P$.
By Proposition~\ref{nn} we know that $P\in \wt X_F$.
By Proposition~\ref{description}, there are exactly two  positive cones in $(D,\vt)$ over $P$ and therefore, by 
Theorem~\ref{morco}, there are   exactly two  positive cones in $(A,\s)$ over $P$, necessarily $\CP$ and
$-\CP$. Thus  we only need to show that $\CC_P(\CM^\eta_P(A,\s))$ is a positive cone
  in $(A,\s)$ over $P$. Upon identifying 
  $(A,\s)$ with $(M_\ell(D), \ad_\Phi) = (M_\ell(D), \Int(\Phi) \circ \vt^t)  $ (without loss of generality) and
  using the scaling from Proposition~\ref{prop-scaling},
  it suffices to show that $\Phi \CC_P(\CM^\eta_P(A,\s))$ is a  positive
  cone on $(M_\ell(D),\vt^t)$ over $P$. 
  A direct computation shows that $\Phi \CC_P(\CM^\eta_P(A,\s)) =
  \CC_P(\CM^{\Phi\eta}_P(M_\ell(D),\vt^t))$, which is equal to
  $\PSD_\ell(\CM^{g(\Phi\eta)}_P(D,\vt))$ by Lemma \ref{CPSD}, which in turn is
  maximal by Proposition~\ref{description} and Proposition~\ref{correspondences}.
  
Finally, for $\CP\in X_{(A,\s)}$, we show that $\CP \cap A^\x = \ve \CM^\eta_P(A,\s) \setminus \{0\}$
for some $\ve\in\{-1,1\}$. Only the left to right inclusion is not obvious. Let $u \in \CP \cap A^\x$.
  Since $\CP = \pm \CC_P(\CM^\eta_P(A,\s))$ we assume for instance that $u
  \in \CC_P(\CM^\eta_P(A,\s))$.  By Lemma~\ref{bleuarg}, $\qnd{u}_\s$ is
  $\eta$-maximal. But $u$ is invertible, so $\qnd{u}_\s = \qf{u}_\s$ is
  $\eta$-maximal, i.e. $u \in \CM^\eta_P(A,\s)$.   
\end{proof}

An immediate consequence of this theorem is the following result, corresponding to the classical bijection between
orderings on $F$ and signatures of quadratic forms with coefficients in $F$: 

\begin{cor}\label{bij} 
The map 
\begin{align*}
X_{(A,\s)} &\to \{\sign_Q^\mu \mid Q\in \wt X_F,\ \mu \text{ a tuple of reference forms for } (A,\s)\}\\
\ve \CC_P(\CM^\eta_P(A,\s)) &\mapsto \ve \sign_P^\eta =\sign_P^{\ve\eta}
\end{align*}
is a bijection \textup{(}where $\ve \in \{-1,1\}$ and $P\in \wt X_F$\textup{)}. 
\end{cor}

As another consequence of Theorem~\ref{positive=max}, 
we can now clarify the hypothesis used in Proposition~\ref{easy}$(1)$:

\begin{cor}\label{1-pos}  
Let $P\in X_F$. The following statements are equivalent: 
\begin{enumerate}[$(1)$]
\item $\s$ is positive at $P$;
\item $|\sign_P^\eta \qf{1}_\s|=n_P$;
\item there exists a prepositive cone $\CP$ on $(A,\s)$ over $P$ such that $1 \in \CP$.
\end{enumerate}
\end{cor}

\begin{proof} The first two statements are equivalent by \cite[Corollary~4.6]{A-U-PS} and the last two statements are
equivalent by Theorem~\ref{positive=max}.
\end{proof}

\subsection{Formally real algebras with involution}

In this subsection we prove the positive cone analogue of the classical 
Artin-Schreier theorem for fields, which states that
$F$ is formally real if and only if $0$ is not  a nontrivial sum of squares
(see Theorem~\ref{fr} $(1)\Leftrightarrow (3)$).

\begin{lemma}\label{sta}
  The following statements are equivalent:
  \begin{enumerate}[$(1)$]
    \item\label{sta1} There is $d \in \Sym(D,\vt)^\x$ such that
      $\qf{d}_\vt$ is strongly anisotropic;
    \item\label{sta2} There is $a \in \Sym(M_\ell(D),\vt^t)^\x$ such that $\qf{a}_{\vt^t}$
      is strongly anisotropic.
  \end{enumerate}
\end{lemma}

\begin{proof}
 $ \eqref{sta1} \Rightarrow \eqref{sta2}$: Take $a = \diag(d,\ldots, d)$. The result
  follows since $g(\qf{a}_{\vt^t})=\ell \x \qf{d}_\vt$ and Morita equivalence preserves
  (an)isotropy. 

 $ \eqref{sta2} \Rightarrow \eqref{sta1}$:  Let $d_1, \ldots, d_\ell \in \Sym(D,\vt)^\x$
 such that $g(\qf{a}_{\vt^t})=\qf{d_1, \ldots, d_\ell}_\vt$. By hypothesis and Morita theory, 
 $\qf{d_1, \ldots, d_\ell}_\vt$ is strongly anisotropic.  It
  follows that the form $\qf{d_1}_\vt$ is also strongly anisotropic.
\end{proof}

\begin{thm}\label{fr}
  The following statements are equivalent:
  \begin{enumerate}[$(1)$]
    \item\label{fr1} $(A,\s)$ is formally real;
    \item\label{fr3} There is $a \in \Sym(A,\s)^\x$ and $P \in X_F$ such that 
    $\CC_P(a) \cap -\CC_P(a)=\{0\}$; 
    \item\label{fr4} There is $b \in \Sym(A,\s)^\x$ such that $\qf{b}_\s$ is
      strongly anisotropic.
  \end{enumerate}
\end{thm}

\begin{proof}
  $\eqref{fr1} \Leftrightarrow \eqref{fr3}$ is clear, so we prove the equivalence of
  \eqref{fr1} and \eqref{fr4}.

By definition $(A,\s)$ is formally real if and only if $X_{(A,\s)}\not= \varnothing$, which is equivalent to
 $\wt X_F \not = \varnothing$,   by Proposition~\ref{nn}.

  $\eqref{fr1} \Rightarrow \eqref{fr4}$: Let $P \in \wt X_F$, and let $d \in
  \Sym(D,\vt)^\x$ such that $\sign^\eta_P \qf{d}_\vt = m_P(D,\vt)$.
  Assume that $\qf{d}_\vt$ is weakly isotropic, i.e. $0 = \sum_{i=1}^k
  \s(x_i)dx_i$, for some $k \in \N$, $x_i \in D^\x$. Since $D$ is a division
  ring, we have $\s(x_1)dx_1 \not = 0$ and thus $k \ge 2$. Then $-d  =
  \sum_{i=2}^k \s(x_ix_1^{-1})dx_ix_1^{-1}$, and in particular
  $\qf{-d,b_2,\ldots,b_{k-1}}_\vt \simeq \qf{d,\ldots,d}_\vt$ for some
  $b_2,\ldots,b_k \in \Sym(D,\vt)^\x$. We obtain a contradiction by taking
  signatures on both sides, since $d$ has maximal signature.

  Therefore there is $d \in \Sym(D,\vt)^\x$ such that $\qf{d}_\vt$ is
  strongly anisotropic, and the result follows by Lemma~\ref{sta}
  and diagram~\eqref{diagram}.

  $\eqref{fr4} \Rightarrow \eqref{fr1}$: Assume that $(A,\s)$ is not formally real. Then $\wt X_F = \varnothing$ and
by Pfister's local-global principle \cite[Theorem~4.1]{LU1}, every hermitian form over $(A,\s)$ is weakly hyperbolic, and
  in particular weakly isotropic, a contradiction.
\end{proof}

\begin{remark}
  In Theorem~\ref{fr}, the element $a$ in statement \eqref{fr3} obviously belongs
  to a prepositive cone on $(A,\s)$. However, the element $b$ from
  statement \eqref{fr4} may be chosen so that it does not belong to any prepositive cone, as the following
  example shows.   Let $(A,\s) = (M_n(F),t)$ and let $b \in \Sym(M_n(F),t)^\x$ be such that
  $\qf{b}_t$ is strongly anisotropic. Let
  $q$ be the quadratic form over $F$ with Gram matrix  $b$. Assume that $b$
  belongs to some  positive cone $\CP$ over $P \in X_F$. 
   Thus 
   $\CP
  = \varepsilon  \CC_P(\CM^\eta_P(M_n(F),t))$ for some $\varepsilon \in \{-1,1\}$, and $\sign^\eta_P
  \qf{b}_{t} =  \varepsilon m_P(M_n(F),t)$ by Theorem~\ref{positive=max}.  
  Since $\sign_P^\eta \qf{1}_t=n$, it follows that $m_P(M_n(F),t)=n$.
  By definition of $\sign_P^\eta$ there exists $\delta \in \{-1,1\}$ such that   
  $\sign_P q  =\sign_P^\eta \qf{b}_t= \delta n$, i.e. $q$ is definite at $P$.
  If this can be done for every $b \in \Sym(A,\s)^\x$, i.e. for every nonsingular
  quadratic form over
  $F$, we obtain that  every strongly anisotropic nonsingular quadratic form over $F$ is definite, i.e.
  that $F$ is a SAP field (cf. \cite{ELP} and \cite{Prestel73}), but not every formally real field is SAP.
\end{remark}

 The analogy between the field and algebra with involution cases presented in Theorem~\ref{fr}
 can be completed by the following result:

\begin{prop}\label{tor_fr}
 The following statements are equivalent:
  \begin{enumerate}[$(1)$]
    \item  $(A,\s)$ is formally real;
    \item $W(A,\s)$ is not torsion;
    \item $X_F\not=\Nil[A,\s]$.
 \end{enumerate}
\end{prop}

\begin{proof} $(1)\Rightarrow (2)$: This is Proposition~\ref{not_torsion}. $(2)\Rightarrow (3)$: By 
Pfister's local-global principle \cite[Theorem~4.1]{LU1}. $(3)\Rightarrow (1)$: By Proposition~\ref{nn}.
\end{proof}

\subsection{Intersections of positive cones}

Recall that Hilbert's 17th problem asks if a polynomial that takes only non-negative values over
the reals can be represented as a sum of squares of rational functions. This question was answered
in the affirmative by E.~Artin and a key step in his proof consists of showing that the totally
positive elements in a field are precisely the sums of squares. In \cite{P-S} Procesi and Schacher
considered noncommutative generalizations of Artin's result, and in \cite{A-U-PS} we proved a
general sums-of-hermitian squares version of one of the main results in their paper:

\begin{thm}[{\cite[Theorem~3.6]{A-U-PS}}]    \label{PSthm}
  Let $b_1,\ldots,b_t \in F^\x$ and consider the Harrison set
  $Y = H(b_1,\ldots,b_t)$. Let $a \in \Sym(A,\s)^\x$ be $\eta$-maximal on $Y$. Then the following
  two statements are equivalent for $u \in \Sym(A,\s)$:
  \begin{enumerate}[$(1)$]
    \item $u$ is $\eta$-maximal on $Y$;
    \item There is $s \in \N$ such that $u \in D_{(A,\s)} (s \x
      \pf{b_1,\ldots,b_t}\ox \qf{a}_\s)$.
  \end{enumerate}
\end{thm}

Our proof uses signatures of hermitian forms and is in essence the same as the one in the field case based on 
Pfister's local-global principle and is straightforward when $A$ is an $F$-division algebra. 

\begin{lemma}\label{max-Q}
 Let $Y = H(b_1,\ldots,b_t)$ and let $a \in \Sym(A,\s)^\x$ be such that
 $a \in \CC_P(\CM^\eta_P(A,\s))$ for every $P \in Y$ \tu{(}i.e.  $a$ is $\eta$-maximal on $Y$ by 
 Lemma~\ref{bleuarg}\tu{)}.
 Let $u \in \Sym(A,\s)$. The following two statements are equivalent:
 \begin{enumerate}[$(1)$]
   \item $u$ is $\eta$-maximal on $Y$;
   \item For every $\CQ \in X_{(A,\s)}$  with $\CQ_F\in Y$, $u \in \CQ$ if and only if $a \in \CQ$.
 \end{enumerate}
\end{lemma}

\begin{proof}
 $(1) \Rightarrow (2)$: By Theorem~\ref{positive=max},
 the only two prepositive cones over any $Q \in \wt X_F$ are are $-\CC_Q(\CM^\eta_Q(A,\s))$ and 
 $\CC_Q(\CM^\eta_Q(A,\s))$.
 Since $a$ is $\eta$-maximal on $Y$, the following are equivalent, for $\CQ \in
 X_{(A,\s)}$ with $\CQ_F \in Y$:
 \begin{align*}
   a \in \CQ &\Leftrightarrow \CQ = \CC_Q(\CM^\eta_Q(A,\s)) \\
       &\Leftrightarrow u \in \CQ \ \ (\text{since $u$ is $\eta$-maximal on } Y).
 \end{align*}

 $(2) \Rightarrow (1)$: Let $P \in Y$. By hypothesis $a \in \CC_P(\CM^\eta_P(A,\s))$ and by (2) we have
 $u\in    \CC_P(\CM^\eta_P(A,\s))$, i.e. $u$ is $\eta$-maximal at $P$ by Lemma~\ref{bleuarg}.
\end{proof}

We can now reformulate Theorem~\ref{PSthm} in terms of intersections of positive cones. Observe that the element $a$ 
in the next theorem plays the role of the element $1$ in the field case.

\begin{thm}\label{intersection}
Let $b_1,\ldots,b_t \in F^\x$,
  let $Y = H(b_1,\ldots,b_t)$ and let $a \in \Sym(A,\s)^\x$ be such that, for
  every $\CP \in X_{(A,\s)}$ with   $\CP_F \in Y$, $a \in \CP\cup -\CP$. Then
  \[\bigcap \{\CP \in X_{(A,\s)} \mid \CP_F \in Y \textrm{ and } a \in \CP\} =
  \bigcup_{s \in \N} D_{(A,\s)} (s \x \pf{b_1,\ldots,b_t} \ox\qf{a}_\s).\]
\end{thm}

\begin{proof}
  Let $u \in \Sym(A,\s)\setminus \{0\}$. Observe that
  \begin{equation}\label{equ}
u \in \bigcap \{\CP \in X_{(A,\s)} \mid \CP_F \in Y \textrm{ and } a
    \in \CP\}
  \end{equation}  
    if and only if for every $\CQ \in X_{(A,\s)}$ such that $\CQ_F \in Y$, ($u \in \CQ \Leftrightarrow
    a \in \CQ$). Indeed, assume that \eqref{equ} holds and
let $\CQ \in X_{(A,\s)}$ be such that $\CQ_F \in Y$. If $u \in \CQ$ but $a
  \not \in \CQ$, then $a \in -\CQ$ and thus $u \in -\CQ$, contradicting that
  $\CQ$ is proper. If $a \in \CQ$ then $u \in \CQ$ by choice of $u$. The converse is
immediate.

  Using this observation and Lemma~\ref{max-Q}, we obtain that $u \in \bigcap \{\CP
  \in X_{(A,\s)} \mid \CP_F \in Y \textrm{ and } a \in \CP\}$ if and only if $u$
  is $\eta$-maximal on $Y$, and the result follows by Theorem~\ref{PSthm}.
\end{proof}

Under stronger conditions, we obtain the following characterization of sums of hermitian squares in $(A,\s)$,
which reduces to Artin's theorem in the case of fields:

\begin{cor}\label{Artin}
Assume that for every $\CP \in X_{(A,\s)}$, $1\in \CP\cup -\CP$. Then
  \[\bigcap \{\CP \in X_{(A,\s)} \mid 1\in \CP\} =
  \Bigl\{\sum_{i=1}^s \s(x_i)x_i \,\Big|\, s\in \N,\ x_i \in A\Bigr\}.\]
\end{cor}

\begin{remark} The hypothesis of Corollary~\ref{Artin} is exactly $X_\s = \wt X_F$ in the terminology of 
\cite{A-U-PS}. More precisely, this property characterizes the algebras with involution for which there is a positive answer 
to the (sums of hermitian squares version of the) 
question formulated by Procesi and Schacher in \cite[p.~404]{P-S}, cf. \cite[Section~4.2]{A-U-PS}.
\end{remark}

\begin{remark} In the field case, results such as Theorem~\ref{intersection} are a direct consequence of the fact
that if $c\in F$ does not belong to some preordering $T$, then $T[-c]$ is a proper preordering. In our setting of prepositive cones on $F$-algebras with involution such a direct approach does not seem to be obvious. 
\end{remark}

\section{Signatures at positive cones}\label{sec:signp}

  Recall from \cite{A-U-Kneb} that to define a signature map from $W(A,\s)$ to $\Z$ at 
	an ordering of $X_F$, we had to introduce a reference form $\eta$ in order to resolve a sign 
	choice at each $P \in X_F$. This sign choice is reflected in the structure of
  $X_{(A,\s)}$ in the fact that there are exactly two  positive cones over $P \in
  \wt X_F$, $\CP$ and $-\CP$ (if $\CP$ denotes one of them).
	
	In this section we define, directly out of $\CP\in X_{(A,\s)}$ over $P\in \wt X_F$, 
	the signature at $\CP$ of  hermitian forms, which we denote by $\sign_\CP$,  
	and show that it coincides with $\sign_P^\eta$ or $-\sign_P^\eta$, cf. 
	Proposition~\ref{equal_sign}. 
	In contrast to $\sign_P^\eta$, the
  definition of $\sign_\CP$ does not require a sign
  choice since $\CP$ and $-\CP$ are different points in $X_{(A,\s)}$ and
  thus both sign choices  occur, one for each of the two positive cones over 
  $P$.

\subsection{Real splitting and maximal symmetric subfields}

Let $E$ be a division algebra. It is well-known that any maximal subfield of $E$ is a splitting field of $E$. 
We are interested in particular in the situation where $E$ is equipped with a positive involution $\tau$. 
In this case,
any maximal symmetric subfield of $(E, \tau)$ is a \emph{real} splitting field of $(E,\tau)$, cf. 
Theorem~\ref{real-splitting} and Definition~\ref{def:rs}.

\begin{lemma}\label{quat}
Let $E$ be a division algebra with centre $K$ and let 
$K\subseteq L \subseteq M$ be subfields of $E$,
where $M$ is maximal.   Then $E\ox_K L$ is Brauer equivalent to a quaternion
division algebra over $L$ if and only if $[M:L]=2$.
\end{lemma}

\begin{proof} Let $k=[L:K]$ and $\ell=[M:L]$, then by \cite[Chapter~1, \S2.9, Prop.,
p.~139]{Plat-Yan},
\[E\ox_K L \simeq C_E(L) \ox_{L} M_k(L),\]
where $C_E(L)$ denotes the centralizer of $L$ in $E$.
Thus,
$\dim_{K} E = \dim_{L} C_E(L) \cdot k^2$.
Also,
$E\ox_K M \simeq  M_{k\ell}(M)$
since $M$ is a splitting field of $E$, cf. \cite[Theorem~2, p.~139]{Plat-Yan}.
Thus,
$\dim_{K} E = (\ell k)^2$.
It follows that
$\dim_{L} C_E(L) =\ell^2$
and thus $C_E(L)$ is a quaternion algebra over $L$ if and only if $\ell=2$.
Since $E\ox_K L$ is Brauer equivalent to $C_E(L)$, the result follows.
\end{proof}

\begin{lemma}\label{almost}
   Let $(E,\tau)$ be an $F$-division algebra with involution and
  let $L$ be a subfield of $\Sym(E,\tau)$, maximal for inclusion and containing $F$. Then either $L$
  is a maximal subfield of $E$ or there is $u \in \Skew(E,\tau)$ such that
  $u^2 \in L$ and $L(u)$ is a maximal subfield of $E$.
\end{lemma}

\begin{proof} 
  We assume that $L$ is not a maximal subfield of $E$. Then $L \subsetneqq C_E(L)$. Let
  $v \in C_E(L) \setminus L$,  $w = (v+\tau(v))/2$ and $u = (v-\tau(v))/2$.
  Since $L \subseteq \Sym(E,\tau)$ we have $\tau(C_E(L)) \subseteq C_E(L)$ and so
  $w,u \in C_E(L)$. Since $w+u = v \not \in L$ we have $w \not \in L$
  or $u \not \in L$. If $w \not \in L$, we get a contradiction since $L(w)
  \subseteq \Sym(E,\tau)$. So $u \not \in L$, but $u^2 \in L$ (because $u^2 \in
  C_E(L)$ and is symmetric, so $L(u^2)$ is a subfield containing $L$ and included
  in $\Sym(E,\tau)$, so equal to $L$).

  We now check that $L(u)$ is a maximal subfield of $E$. Assume that this is not the
  case. As above there is $x \in C_E(L(u)) \setminus L(u)$, and since $\tau(L(u))
  = L(u)$ we have $\tau(x) \in C_E(L(u))$ and thus $x$ can be
  written as $y+z$ with $\tau(y) = y$, $\tau(z)=-z$ and $y,z \in C_E(L(u))$. Since $x
  \not \in L(u)$ we have either $y \not \in L(u)$ or $z \not \in L(u)$. If $y
  \not \in L(u)$ then $L \subsetneqq L(y) \subseteq \Sym(E,\tau)$, impossible. So
  $z \not \in L(u)$. Then $uz=zu$ (since $z \in C_E(L(u))$), $uz \not \in L$
  (since $z \not \in L(u)$) and $uz$ is symmetric. Therefore $L \subsetneqq L(uz)
  \subseteq \Sym(E,\tau)$, a contradiction.
\end{proof}

\begin{lemma}\label{pos_ext}
  Let $P \in X_F$ and let $(B,\tau)$ be an $F$-algebra with involution such that  $\tau$  is positive at $P$. 
  Let $L$ be a subfield of $\Sym(B,\tau)$, maximal
  for inclusion and containing $F$. Then $P$ extends to an ordering on $L$.
\end{lemma}
\begin{proof}
Since $\tau$ is positive at $P$,  by Corollary~\ref{1-pos} there is $\CP \in X_{(B,\tau)}$ over $P$ such that $1\in \CP$
and it follows that $P \subseteq \CP$ by Proposition~\ref{easy}. Since, for
  $p \in P$ and $x \in L$, $px^2 = \tau(x)px$, the preordering generated by $P$
  in $L$ is contained in $\CP$, so is proper, and is therefore included in an
  ordering on $L$. 
\end{proof}

\begin{prop}\label{Z}
Let $(E, \tau)$ be an $F$-division algebra with involution of the second kind. Let $d\in F^\x$ be
such that $Z(E)=F(\sqrt{-d})$.
Then 
\[\Nil[E,\tau]= H(-d)=\{P\in X_F \mid P\text{ extends to } Z(E)\},\]
where $H(-d)$ denotes the usual Harrison set. 
\end{prop}

\begin{proof} 
 Observe that for $P\in X_F$,
 \begin{equation}\label{s-ns}
E\ox_F F_P \cong E\ox_{Z(E)} (Z(E)\ox_F F_P) \cong \begin{cases}
E\ox_{Z(E)} F_P(\sqrt{-d}) \sim F_P(\sqrt{-1}) & \text{if } d\in P\\
E\ox_{Z(E)} (F_P\x F_P)  & \text{if } d\in -P 
\end{cases}.
\end{equation}
Assume first that $P$ extends to $Z(E)$, i.e. that $-d\in P$. By \eqref{s-ns}, $E \ox_F F_P$ is an
$F_P$-algebra which is not
simple. By \cite[Lemma~2.1$(iv)$]{A-U-Kneb} $W(E\ox_F F_P, \tau_P)=0$, implying that $P \in
\Nil[E,\tau]$. 
Conversely, assume that $P\in \Nil[E,\tau]$, but that $d\in P$. Then $E\ox_F F_P \sim
F_P(\sqrt{-1})$. Using \cite[(2.1)]{A-U-stab} this
contradicts  \cite[Proposition~2.2(1)]{A-U-stab}. 
\end{proof}

\begin{thm}\label{real-splitting}
  Let $(E,\tau)$ be an $F$-division algebra with involution, and let $L$ be a
subfield of $\Sym(E,\tau)$, maximal for inclusion and containing $F$. Then one of the following
holds:
  \begin{enumerate}[$(1)$]
\item $L$ is a maximal subfield of $E$ and so $E \ox_F L \sim L$. This can only occur when $\tau$ is
of the first kind.
\item There exists $u\in \Skew(E,\tau)^\x$ such that $u^2 \in L$ and $L(u)$ is a maximal subfield of
$E$.
    	\begin{enumerate}[$(a)$]
	\item If $\tau$ is of the first kind, $E \ox_F L$ is Brauer equivalent
     to a quaternion division algebra $(u^2,c)_L$ for some $c\in L$.
\item If $\tau$ is of the second kind and $Z=F(\sqrt{-d})$, we may take $u=\sqrt{-d}$ and $E \ox_F
L\sim L(\sqrt{-d})$.
	\end{enumerate}
  \end{enumerate}
  Furthermore, if $\tau$ is positive at $P\in X_F$, then there is an ordering $Q$ on
$L$ that extends $P$, $\tau$ is orthogonal in $(1)$ and $\tau$ is symplectic in $(2a)$. In addition, for every
ordering $Q$ on $L$ extending $P$, we have $u^2, c\in -Q$ in $(2a)$, and $d\in Q$
  in $(2b)$.
\end{thm}

\begin{proof} If $L$ is a maximal subfield of $E$, we must have $Z(E)=F$ since $L\subseteq
\Sym(E,\tau)$, and so
$\tau$ is of the first kind and  $E\ox_F L \sim L$. 

If $L$ is not a maximal subfield of $E$, there exists $u\in \Skew(E,\tau)^\x$ such that $u^2 \in L$
and $L(u)$ is a maximal
subfield  of $E$ by Lemma~\ref{almost}.  

Suppose that $\tau$ is of the first kind, so that $Z(E)=F$. 
Since $[L(u):L]=2$, $E\ox_F L$ is Brauer equivalent to a quaternion division algebra $B$ over $L$ by
Lemma~\ref{quat}.
Since $L(u)$ is a maximal subfield of $E$, we have $(E\ox_F L) \ox_L L(u)\cong E\ox_F L(u)\sim L(u)$
and so
$B\ox_L L(u)\sim L(u)$. By \cite[Chapter~III, Theorem~4.1]{Lam}, there exists $c\in L$ such that
$B \cong (u^2, c)_L$.

Suppose that $\tau$ is of the second kind. Then $Z(E)=F(\sqrt{-d})$ for some $d\in F^\x$ and
$\tau(\sqrt{-d})= - \sqrt{-d}$.
Since $\sqrt{-d} \in L(u)\setminus L $ and $[L(u):L]=2$, we have $L(u)=L(\sqrt{-d})$ and we take
$u=\sqrt{-d}$.
Since $L(\sqrt{-d})$ is a maximal subfield of $E$, it follows that
\[E\ox_F L \cong  E\ox_{Z(E)} (Z(E)\ox_F L)\cong E\ox_{Z(E)} (F(\sqrt{-d})\ox_F L)\cong E\ox_{Z(E)} L(\sqrt{-d})
\sim L(\sqrt{-d}).\]

Finally, assume that $\tau$ is positive at $P$. In particular, $P\not\in \Nil[E,\tau]$.
By Lemma~\ref{pos_ext} there exists an ordering $Q$ on $L$ that extends 
$P$. In particular, $Q\not\in \Nil[E\ox_F L, \tau\ox\id_L]$. By
Remark~\ref{nil-and-non}  it follows that
$\tau$ must be orthogonal in
$(1)$ and symplectic in $(2a)$. Furthermore, for any ordering $Q$ on $L$ that extends $P$ we have:
in $(2a)$, $(u^2,c)_L \ox_L L_Q$ must be division and
so $u^2,c \in -Q$;
in $(2b)$, if $-d\in Q$, then $P$ extends to $Z(E)$ and so $P\in \Nil[E,\tau]$ by Proposition~\ref{Z},
contradiction.
\end{proof}

\begin{defi}\label{def:rs} 
Let $(B,\tau)$ be an $F$-algebra with involution and let $P\in X_F$. An 
ordered
extension $(L,Q)$ of $(F,P)$ is called a \emph{real splitting field} of $(B,\tau)$ and we say that $(L,Q)$ \emph{real
splits} $(B,\tau)$  if $B\ox_F L$ is Brauer
equivalent
to one of the $L$-algebras $L$, $L(\sqrt{-d})$ or $(-a,-b)_L$, where $a,b,d\in Q$ and the involution 
$\tau\ox\id$ is positive at $Q$. 
\end{defi}

\begin{remark} The condition that $\tau\ox\id$ be positive at $Q$ is equivalent to $\tau$ being positive at $P$ and is also
equivalent to $\tau\ox\id$ being adjoint to a quadratic or hermitian form that is positive definite at $Q$.
\end{remark}

\begin{prop}\label{Galois-1} 
Let $(E,\tau)$ be an $F$-division algebra with involution and let $P\in X_F$. Assume that $\tau$ is positive at $P$.
Let $L$ be a subfield of $\Sym(E,\tau)$, maximal for inclusion and containing $F$. 
Then there is a Galois
extension $N$ of $F$ such that $L\subseteq N$,
$P$ extends to $N$ and for every extension $Q$ of $P$ to $N$,
we have that $(N,Q)$ is a real splitting field of $(E, \tau)$. In addition, $[N:F] \leq (\deg E)!$.
\end{prop}

\begin{proof}
  Observe that, under the hypothesis for $L$ and $\tau$ and by
  Theorem~\ref{real-splitting}, $E \ox_F L \cong M_t(E_0)$ where $E_0$ is one of
  $L$, $L(\sqrt{-d})$, $(-a,-b)_L$, and for every ordering $P'$ on $L$ extending
  $P$, $d,a,b \in P'$. Moreover $P$ extends to $L$, so we fix a real closure
  $F_P$ of $F$ at $P$ such that $L \subseteq F_P$.

  Write $L=F(a)$ for some $a \in \Sym(E,\tau)$.  Since $\tau$ is positive at $P$,
  $(E\ox_F F_P, \tau\ox\id)$ is isomorphic to a matrix algebra over one of $F_P$,
  $F_P(\sqrt{-1})$ or $(-1,-1)_{F_P}$, equipped with the (conjugate) transpose
  involution. We denote this isomorphism by $\lambda$. Then the matrix
  $\lambda(a\ox 1)$ is symmetric with respect to conjugate transposition. 
  Thus $\lambda(a\ox 1)$ is congruent to a diagonal matrix with entries in $F_P$,
  cf. \cite[Theorem~9]{Lee-1949} and all the roots of the reduced
  characteristic polynomial of $a$ (which has coefficients in $F$) are in $F_P$.
  Hence all the roots of the minimal polynomial $\min_a$ of $a$ over $F$ are in
  $F_P$. Let $N$ be the splitting field of $\min_a$ in $F_P$. Then $L \subseteq
  N$ and  $N/F$ is a Galois extension. Since $N \subseteq F_P$, the ordering $P$
  extends to $N$.

  Now let $Q$ be any ordering on $N$ extending $P$. Then $P':=Q \cap L$ is an
  ordering on $L$ extending $P$ and, as mentioned above, $E \ox_F L \cong M_t(E_0)$ with  
  $d,a,b \in P'$. Therefore $E \ox_F N \cong M_t(E_0)$ with $d,a,b \in
  Q$.
  
  Concerning $[N:F]$, observe that $F(a)$ is a subfield of $E$ and therefore $[F(a):F]\leq \deg E$ by
  the Centralizer Theorem, cf. \cite[\S2.9]{Plat-Yan}. The conclusion follows. 
\end{proof}

\subsection{The signature at a positive cone}\label{subsec:signp}
Let $\CP \in X_{(A,\s)}$ be a positive cone over $P \in X_F$. Recall that $P \in \wt X_F$ 
by Theorem~\ref{positive=max}.
Our first objective is to show the following:

\begin{thm}\label{sylvester-decomp}
  Let $h$ be a nonsingular 
  hermitian form over $(A,\s)$. Then there are $c \in \CP\cap A^\x$ 
  and $\beta_1, \ldots, \beta_t \in P^\x$ such that
  \[n_P^2 \x \qf{\beta_1, \ldots, \beta_t}\ox h \simeq (\qf{a_1,\ldots,a_r} \perp
  \qf{b_1,\ldots,b_s}) \ox
  \qf{c}_\s\]
  with $a_1, \ldots, a_r \in P^\x$, $b_1, \ldots, b_s \in -P^\x$, $t\leq 2^{n_P}(\deg D)!$  and where $n_P$
  is the matrix size of $A\ox_F F_P$. 
\end{thm}

Observe that the hypothesis that $h$ is nonsingular in Theorem~\ref{sylvester-decomp} is not
restrictive, cf.~Proposition~\ref{ns-again}. 
In order to prove Theorem~\ref{sylvester-decomp} we need the following three lemmas.

\begin{lemma}\label{pre-sylvester}
  Assume $(D, \vt) \in \{(F,\id), (F(\sqrt{-d}),-), ((-a,-b)_F,-)\}$ where $a,b,d \in P^\x$
  and let  
  $h$ be a nonsingular hermitian form over $(M_\ell (D), \vt^t)$. Then
  \[\ell^2 \x h \simeq (\qf{a_1,\ldots,a_r} \perp \qf{b_1,\ldots,b_s} )\ox \qf{I_\ell}_{\vt^t}\]
  for some $a_1, \ldots, a_r \in P^\x$ and $b_1, \ldots, b_s \in -P^\x$.
\end{lemma}
\begin{proof}
  Recall that by  Lemma~\ref{morita-diag}, the hermitian form $\ell \x h$ is diagonal
with coefficients in $\Sym(M_\ell (D), \vt^t)^\x$, so we can assume without loss of generality that
$h = \qf{\alpha}_{\vt^t}$ with
  $\alpha \in \Sym(M_\ell (D), \vt^t)^\x$. 
  Observe that
  \[X_{(D,\vt )} = \{R,-R \mid R \in X_F\}\]
  since $\Sym(D,\vt)=F$
  and thus, by  Proposition~\ref{correspondences},
  \[X_{(M_\ell (D), \vt^t)} = \{\PSD_\ell (R), -\PSD_\ell (R) \mid R \in X_F\}.\]
  Without loss of generality we can assume that $\CP = \PSD_\ell(P)$.
   Diagonalizing the matrix $\alpha$ by congruences, we obtain 
   $h \simeq \qf{\alpha'}_{\vt^t}$ where $\alpha' = \diag (u_1,\ldots,u_\ell)$ with $u_1, \ldots,
  u_\ell \in F^\x$. Let $g$ be the collapsing Morita equivalence from  \eqref{diagram}, then
  $g(h) =   \qf{u_1,\ldots,u_\ell}_\vt$ and
  \[\ell \x g(h) \simeq \qf{u_1,\ldots,u_1}_\vt \perp \cdots \perp \qf{u_\ell,\ldots,
  u_\ell}_\vt.\]
  Applying $g^{-1}$ we obtain
\[\ell \x h \simeq \qf{u_1 I_\ell, \ldots, u_\ell I_\ell}_{\vt^t} = \qf{u_1,\ldots, u_\ell}\ox
\qf{I_\ell}_{\vt^t}.\qedhere\]
 \end{proof}

Let $L/F$ be a finite field extension. The trace map $\Tr_{L/F}$
induces an $A$-linear homomorphism $\Tr_{A\ox_F L} =\id_A\ox \Tr_{L/F}: A\ox_F L \to A$. Let $(M,h)$ be a 
hermitian form over $(A\ox_F L, \s\ox\id)$. The \emph{Scharlau transfer} $\Tr_*(M,h)$ is defined to be the hermitian
form   $(M,\Tr_{A\ox_F L}\circ h)$ over $(A,\s)$. Let $q$ be a quadratic form over $L$, then \emph{Frobenius reciprocity}, i.e.
\[\Tr_*(q\ox (h\ox L)) \simeq \Tr_*(q)\ox h\]
(where  $\Tr_*(q)$ is the usual Scharlau transfer of $q$ with respect to $\Tr_{L/F}$)
 can be proved just as for quadratic forms, 
  cf. \cite[VII, Theorem~1.3]{Lam}.

For $L/F$ as above and $P\in X_F$, we denote the set of orderings on $L$ that extend $P$ by
$X_L/P$.

\begin{lemma}\label{sylvester-split}
  Assume that $\s$ is positive at $P$ and that $D \in \{F , F(\sqrt{-d}) ,
  (-a,-b)_F\}$ where $a,b,d \in P^\x$. Let $h$ be a nonsingular hermitian form
  over $(A,\s)$. Then there are $a_1, \ldots, a_r \in P^\x$,  $b_1,
  \ldots, b_s \in -P^\x$ and $\alpha_1,\ldots, \alpha_{2^k} \in F^\x$ such that
  \[\ell^2 \x \qf{\alpha_1, \ldots, \alpha_{2^k}}\ox h \simeq (\qf{a_1,\ldots,a_r}
  \perp \qf{b_1,\ldots,b_s})\ox \qf{1}_{\s}, \]
  where $k\leq \ell$.
  Furthermore,  for every ordering $Q \in X_F$
  such that $\s$ is positive at $Q$ and either $d\in Q$ in case $D=F(\sqrt{-d})$, or $a,b\in Q$ in case 
  $D=(-a,-b)_F$, we have $\alpha_1,\ldots, \alpha_{2^k} \in Q$.
\end{lemma}

\begin{proof}
  Let $\s=\ad_\vf$ for some form $\vf$ over $(D, \vt)$, where $\vt$ is as in
  Lemma~\ref{pre-sylvester}. Since $\s$ is positive at $P$, we may assume that
  $\vf=\qf{d_1,\ldots, d_\ell}$ with $d_1,\ldots, d_\ell \in P^\x$. Up to
  renumbering, let $\{d_1,\ldots, d_k\}$ be a minimal subset of $\{d_1,\ldots,
  d_\ell\}$ such that $F(\sqrt{d_1},\ldots, \sqrt{d_k})= F(\sqrt{d_1},\ldots,
  \sqrt{d_\ell})=:L$. Since the extension $L/F$ is obtained by successive proper
  quadratic extensions by elements of $P$ we have $|X_L/P| =2^k$.
  
  After scalar extension to $L$ we obtain
  \[(A\ox_F L, \s\ox\id) \cong (M_\ell(D\ox_F L), \ad_{\vf\ox L}) \cong
  (M_\ell(D\ox_F L), \ad_{\qf{1,\ldots, 1}_\vt})  =(M_\ell(D\ox_F L), \vt^t).\] 
  Observe that $D\ox_F L$ is an element of $\{L , L(\sqrt{-d}) , (-a,-b)_L\}$ and is an
  $L$-division algebra since $a,b,d \in P^\x$. Thus, by
  Lemma~\ref{pre-sylvester} (with $F$ replaced by $L$) there are $u_1,\ldots,
  u_t \in L^\x$ such that
  \[\ell^2 \x (h\ox L) \simeq \qf{u_1,\ldots, u_t} \ox\qf{1}_{\vt^t}.\] 
  Applying the Scharlau transfer induced by the trace map $\Tr_{L/F}$ 
    gives
  \[\ell^2 \x \Tr_*(h\ox L) \simeq \Tr_*( \qf{u_1,\ldots, u_t}\ox \qf{1}_{\vt^t}).\]
  Thus, by Frobenius reciprocity  
  and since $\s\ox\id \cong \vt^t$, we obtain
  \[\ell^2\x \Tr_*(\qf{1})\ox h \simeq \Tr_*( \qf{u_1,\ldots, u_t})\ox \qf{1}_\s.\]
  Let $\qf{\alpha_1,\ldots,\alpha_{2^k}}$ be a diagonalization of the form
  $\Tr_*(\qf{1})$. Since $L/F$ is separable, $\Tr_{L/F}$ is nonzero and it follows from
  \cite[VII, Proposition~1.1]{Lam} that $\Tr_*( \qf{u_1,\ldots, u_t})$ is a nonsingular quadratic form
  over $F$. Therefore we can separate the coefficients of a diagonalization of  $\Tr_*( \qf{u_1,\ldots, u_t})$
  into elements
  of $P^\x$ and elements of $-P^\x$.

  Finally, let $Q \in X_F$ be such that $\s$ is positive at $Q$ and either 
  $d\in Q$ in case $D=F(\sqrt{-d})$, or $a,b\in Q$ in case 
  $D=(-a,-b)_F$. We have $d_1,\ldots,d_\ell \in Q$ and, as above, $|X_L/Q|
  = 2^k$. By the quadratic Knebusch trace formula 
  \cite[Chapter~3,   Theorem~4.5]{Sch}, $\sign_Q(\Tr_*(\qf{1}))=2^k=[L:F]=
  \dim \Tr_*(\qf{1})$, proving that $\alpha_1, \ldots, \alpha_{2^k} \in Q$.
\end{proof}

\begin{lemma}\label{sylvester-positive}
  Let $h$ be a nonsingular hermitian form over $(A,\s)$ where $\s$ is a positive involution
  at $P$. Then there are $\beta_1, \ldots, \beta_t \in P^\x$ such that
  \[n_P^2 \x \qf{\beta_1, \ldots, \beta_t} \ox h \simeq (\qf{a_1,\ldots,a_r} 
  \perp
  \qf{b_1,\ldots,b_s} ) \ox
  \qf{1}_\s\]
  with $a_1, \ldots, a_r \in P^\x$, $b_1, \ldots, b_s \in -P^\x$ and $t\leq 2^{n_P} (\deg D)!$.
\end{lemma}
\begin{proof}
  By  Theorem~\ref{main_pos} there exists a  positive involution on $D$.  Thus, 
  by Proposition~\ref{Galois-1} there exists a finite Galois extension $N$ of  
  $F$ such that $P$ extends to an
ordering $Q$ of $N$ and such that $(N,Q)$ real splits $(A,\s)$. In particular we can take $N\subseteq F_P$. 
Therefore we can apply Lemma~\ref{sylvester-split} to $(A \ox_F N, \s \ox \id)$ and we obtain
  \[M^2 \x \qf{\alpha_1, \ldots, \alpha_{2^k}}\ox (h \ox N) \simeq (\qf{u_1,\ldots,u_r} \perp
  \qf{v_1,\ldots,v_s})\ox \qf{1}_{\s\ox\id},\]
  where $M$ is the matrix size of $A\ox_F N$,
  $u_1, \ldots, u_r
  \in Q^\x$, $v_1,\ldots, v_s \in -Q^\x$, $\alpha_1,\ldots, \alpha_{2^k}$ are as given in Lemma~\ref{sylvester-split}
  and $k\leq M$. 
  Observe that in fact $M=n_P$ (as defined in \eqref{def:np})
  since $(N,Q)$ real splits $(A,\s)$, so that
   the passage from $A\ox_F N$ to 
  $A\ox_F F_P$ does not change the degree of the underlying division algebra.  
  Furthermore, since $\s$ is positive at $P$, $\s$ is positive at every $R \in X_N/P$
  and by Proposition~\ref{Galois-1}, $(N,R)$ is a real splitting field of $(A,\s)$. Therefore, by Lemma~\ref{sylvester-split},
  $\alpha_1,\ldots, \alpha_{2^k} \in R$ for every $R\in X_N/P$.
  
  Applying the Scharlau transfer induced by the trace map $\Tr_{N/F}$ and Frobenius reciprocity 
  give 
  \[n_P^2 \x \Tr_*( \qf{\alpha_1, \ldots, \alpha_{2^k}})\ox h \simeq \Tr_*(\qf{u_1, \ldots, u_r, v_1,\ldots,  v_s})\ox
  \qf{1}_\s.\]
  Observe now that $\dim \Tr_*(\qf{\alpha_1, \ldots, \alpha_{2^k}}) = 2^k [N:F]$ 
  and that, since $N$ is
  an ordered Galois extension of $(F,P)$, the number $|X_N/P|$  of extensions of $P$ to $N$
  is $[N:F]$ (this is a direct consequence of \cite[Corollary~1.3.19]{PD2001}   and the fact that the extension is Galois).
  In particular, and using that $\alpha_1, \ldots, \alpha_{2^k}$ belong to every $R \in X_N/P$, we obtain,
  using the quadratic Knebusch trace formula \cite[Chapter~3, Theorem~4.5]{Sch}, that
  \[\sign_P \Tr_*(\qf{\alpha_1, \ldots, \alpha_{2^k}}) = \sum_{R \in X_N/P} \sign_P \qf{\alpha_1, \ldots, \alpha_{2^k}} = 
  2^k [N:F].\]
  It follows that if we write $\Tr_* (\qf{\alpha_1, \ldots, \alpha_{2^k}}) \simeq \qf{\beta_1, \ldots, \beta_t}$
  with $\beta_1, \ldots, \beta_t \in F^\x$, we must have 
  $\beta_1, \ldots, \beta_t \in P^\x$ since $t=2^k[N:F]$.
  Therefore
  \[n_P^2 \x \qf{\beta_1, \ldots, \beta_t} \ox h \simeq \Tr_*(\qf{u_1, \ldots, u_r, v_1,\ldots,  v_s})\ox \qf{1}_\s.\]
 The result follows and the bound on $t$ is obtained from Proposition~\ref{Galois-1}.
\end{proof}

\begin{proof}[Proof of Theorem~\ref{sylvester-decomp}]
  Let
  $\gamma$ be an involution on $A$, positive at $P$, and let
  $c \in \Sym(A,\s)^\x$ be such that $\gamma = \Int(c^{-1}) \circ \s$. It follows from \cite[Remark~4.3 and Proposition~4.4]{A-U-PS}
  and  Theorem~\ref{positive=max}  that $c\in \CP\cap A^\x$ or $c \in -(\CP\cap A^\x)$.
  
  By  Proposition~\ref{prop-scaling}, $c^{-1} \CP$ is a
  positive cone on $(A,\gamma)$ over $P$. Furthermore, $c^{-1} h$ is a hermitian form
  over $(A,\gamma)$ by scaling. So by
  Lemma~\ref{sylvester-positive}, there are  $\beta_1, \ldots, \beta_t \in P^\x$ such
  that
  \[n_P^2 \x \qf{\beta_1, \ldots, \beta_t} \ox c^{-1} h \simeq (\qf{a_1,\ldots,a_r} \perp
  \qf{b_1,\ldots,b_s} )\ox \qf{1}_\gamma\]
  with $a_1, \ldots, a_r \in P^\x$ and $b_1, \ldots, b_s \in -P^\x$. 
  Scaling by $c$ gives
  \[n_P^2 \x \qf{\beta_1, \ldots, \beta_t} \ox h \simeq (\qf{a_1,\ldots,a_r} \perp
  \qf{b_1,\ldots,b_s}) \ox \qf{c}_\s\]
  and the result follows since  $c\in \CP\cap A^\x$ or $c \in -(\CP\cap A^\x)$.
\end{proof}

The following corollary is an immediate weaker version of
Theorem~\ref{sylvester-decomp}.
\begin{cor}\label{sylvester-simple}
  Let $h$ be a nonsingular hermitian form over $(A,\s)$. Then there exist 
  $\beta_1, \ldots, \beta_t \in P^\x$ (with $t\leq 2^{n_P} (\deg D)!$) ,  $a_1, \ldots, a_r \in \CP\cap A^\x$ and $b_1, \ldots,
  b_s \in -(\CP\cap A^\x)$ such that
  \[n_P^2 \x \qf{\beta_1, \ldots, \beta_t}\ox  h \simeq \qf{a_1, \ldots, a_r}_\s \perp
  \qf{b_1, \ldots, b_s}_\s.\]
\end{cor}

A result in the style of Sylvester inertia is immediate:
\begin{lemma}\label{inertia}
  Let $h$ be a  hermitian form over $(A,\s)$ and suppose
  \begin{align*}
    h & \simeq \qf{a_1, \ldots, a_r}_\s \perp \qf{b_1, \ldots, b_s}_\s \\
      & \simeq \qf{a'_1, \ldots, a'_p}_\s \perp \qf{b'_1, \ldots, b'_q}_\s
  \end{align*}
  with $a_1,\ldots,a_r,a'_1,\ldots,a'_p \in \CP \cap A^\x$ and $b_1, \ldots,b_s, b'_1,
  \ldots, b'_q \in -(\CP\cap A^\x)$. Then $r=p$ and $s=q$.
\end{lemma}
\begin{proof}
  For dimension reasons we have $r+s = p+q$. The result will follow if we show
  $r-s = p-q$, i.e. $r+q = p+s$. Assume for instance that $r+q > p+s$. We have
  the following equality in $W(A,\s)$:
  \[\qf{a_1,\ldots,a_r,-b'_1,\ldots,-b'_q}_\s = \qf{a'_1,\ldots, a'_p, -b_1,
  \ldots, -b_s}_\s.\]
  Relabelling the entries gives 
  $\qf{\alpha_1, \ldots, \alpha_{r+q}}_\s = \qf{\beta_1, \ldots,
  \beta_{p+s}}_\s$,
  with $\alpha_i, \beta_j \in \CP\cap A^\x$. Since $r+q > p+s$ there
  is $i \in \{1, \ldots, r+q\}$ such that
  \[\qf{\alpha_{i+1}, \ldots, \alpha_{r+q}}_\s \simeq \qf{-\alpha_1, \ldots,
  -\alpha_i, \beta_1, \ldots, \beta_{p+s}}_\s.\]
If follows that $-\alpha_1 \in D_{(A,\s)} (\qf{\alpha_{i+1}, \ldots, \alpha_{r+q}}_\s)\subseteq \CP$,
contradicting that $\CP$ is
proper.
\end{proof}

We use Corollary~\ref{sylvester-simple} to define the signature at $\CP$.

\begin{defi}\label{def_sign}
  Let $h$ be a nonsingular hermitian form over $(A,\s)$, $\CP\in X_{(A,\s)}$ 
   and let $t$, $\beta_i$, $a_j$ and $b_k$ be as
  in Corollary~\ref{sylvester-simple}. We define the \emph{signature of $h$ at $\CP$} by
  \[\sign_\CP h := \frac{r-s}{n_Pt}.\]
 \end{defi}

Corollary~\ref{sylvester-simple} suggests that $\sign_\CP h$ should be a multiple of $r-s$ by some constant.
The reason for the particular normalization applied to $r-s$ in the definition above is given by:

\begin{prop}\label{equal_sign}
  Let $\eta$ be a reference form for $(A,\s)$. Then there is $\varepsilon_P \in
  \{-1,1\}$ such that $\sign_\CP = \varepsilon_P \sign^\eta_P$.
\end{prop}

\begin{proof}
  Let $h$ be a nonsingular hermitian form over $(A,\s)$ and consider
   \[n_P^2 \x \qf{\beta_1, \ldots, \beta_t} \ox h \simeq \qf{a_1, \ldots, a_r}_\s \perp
  \qf{b_1, \ldots, b_s}_\s\]
  as in Corollary~\ref{sylvester-simple}. By  Theorem~\ref{positive=max} there exists $\ve_P \in \{-1,1\}$
  such that $\CP\cap A^\x = \ve_P \CM^\eta_P (A,\s)\sm \{0\}$.
  Applying the map $\sign_P^\eta$ to this isometry, and since $\sign^\eta_P \qf{a_i}_\s=\ve_P n_P$, 
  $\sign^\eta_P \qf{b_j}_\s=-\ve_P n_P$
     we obtain
  $n_P^2 t \sign_P^\eta h = \ve_P n_P (r-s)$, which proves the result.
\end{proof}

\begin{remark} Observe that for $\CP \in X_{(A,\s)}$, $\sign_{-\CP} =-\sign_\CP$. Furthermore, 
the following are immediate consequences of Proposition~\ref{equal_sign} and \cite[Theorem~2.6]{A-U-prime}:
let $h_1, h_2$ be hermitian forms over $(A,\s)$, let $q$ be a quadratic
      form over $F$ and let $\CP \in X_{(A,\s)}$ be over $P \in X_F$, then 
      \begin{align*}
      \sign_\CP(h_1)&=0 \text{ if $h_1$ is hyperbolic}, \\
      \sign_\CP(h_1 \perp h_2) &= \sign_\CP h_1 + \sign_\CP h_2,\\
      \sign_\CP( q\ox h_1) &= (\sign_P q) (\sign_\CP h_1).
      \end{align*}
In particular, the map $\sign_\CP$ is well-defined on $W(A,\s)$. 
\end{remark}

After introducing the topology $\CT_\s$ on $X_{(A,\s)}$ we will show in the final section that the total signature map 
$\sign h: X_{(A,\s)}\to \Z,\ \CP \mapsto \sign_\CP h$ is 
continuous, cf. Theorem~\ref{thm:signp}.

\section{The topology of  $X_{(A,\s)}$}\label{sec:top}

One of the main objectives of this section is to show that the natural topology
on $X_{(A,\s)}$ is spectral, cf. Theorem~\ref{thm:spectral}. 
Spectral topologies were studied in great detail by Hochster \cite{Hochster} in order to completely describe the topology on $\mathrm{Spec} (A)$ for 
any commutative ring $A$. We also refer to \cite[Section~6.3]{Marshall_real_spectra} 
and the recent monograph on spectral spaces \cite{DST}.

Let $(X, \CT)$ be a topological space. We denote by $\cB(\CT)$ the set of all
subsets of $X$ that are compact open in $\CT$. (Note that for us compact means
quasicompact.) 

The space $(X, \CT)$ is \emph{spectral} if it is $T_0$ and compact, 
$\cB(\CT)$ is closed under finite intersections and forms an open basis of $\CT$, and every nonempty
irreducible closed subset has a generic point, cf. \cite[Section~0]{Hochster}.

We define
the \emph{patch topology} on $X$, denoted $\CT_\patch$, as the topology with
subbasis
\[\{U, X\sm V \mid U, V \in \cB(\CT)\},\]
cf. \cite[Sections~2 and 8]{Hochster}.
Observe that if  $X$ is compact, this subbasis can be replaced by the subbasis
\[\{U\sm V \mid U, V \in \cB(\CT)\},\]
which is actually a basis of $\CT_\patch$ if $\cB(\CT)$ is in addition closed
under finite unions and finite intersections.
The interesting aspects of a spectral topology derive from the properties of $\CT_\patch$ and its
links with $\CT$.

We define, for $a_1, \ldots, a_k \in \Sym(A,\s)$,
\[H_\s(a_1, \ldots, a_k) := \{\CP \in X_{(A,\s)} \mid a_1, \ldots, a_k \in
\CP\}.\]
We denote by $\CT_\s$ the topology on $X_{(A,\s)}$ generated by the sets
$H_\s(a_1, \ldots, a_k)$, for $a_1, \ldots, a_k \in \Sym(A,\s)$, and by
$\CT_\s^\x$ the topology on $X_{(A,\s)}$ generated by the sets
$H_\s(a_1, \ldots, a_k)$, for $a_1, \ldots, a_k \in \Sym(A,\s)^\x$.
Recall that we denote the usual Harrison topology on $X_F$ or $\wt X_F$ by $\CT_H$.
We first show that the topologies   $\CT_\s$ and $\CT_\s^\x$ are equal.

\begin{lemma}\label{homeomorphisms}\mbox{}
  \begin{enumerate}[$(1)$]
    \item Let $a \in \Sym(A,\s)^\x$. The scaling map $X_{(A,\s)} \rightarrow X_{(A,\Int(a)
      \circ \s)}$, $\CP \mapsto a\CP$ is a homeomorphism in the following two
      cases, where 
      \begin{enumerate}[$(a)$]
      \item $X_{(A,\s)}$ is equipped with $\CT_\s$ and $X_{(A, \Int(a) \circ \s)}$ with $\CT_{\Int(a) \circ \s}$;
      \item $X_{(A,\s)}$ is equipped with $\CT_\s^\x$ and $X_{(A,\Int(a) \circ \s)}$ with $\CT_{\Int(a) \circ \s}^\x$.
      \end{enumerate}

\item The topologies $\CT_{\vt^t}$ and $\CT_{\vt^t}^\x$ on $X_{(M_\ell(D), \vt^t)}$ are equal.

  \end{enumerate}
\end{lemma}
\begin{proof}
(1) Since the inverse of this map is of the same type, it suffices to
      prove that it is continuous. The result is clear in both cases since the
      inverse image of $H_{\Int(a) \circ \s}(a_1,\ldots,a_k)$ is
      $H_\s(a^{-1}a_1,\ldots, a^{-1}a_k)$, and $a^{-1}a_i$ is invertible if and
      only if $a_i$ is.
 
(2) We consider the set $H_{\vt^t}(a)$ for some $a\in \Sym(M_\ell(D), \vt^t)$. By (P3) we may
assume that $a$ is diagonal. Therefore, let
$r\in \N$, let $d_1,\ldots, d_r \in D^\x$ and let $\CP \in X_{(M_\ell(D),\vt^t)}$. Then, using (P2) and (P3), we obtain 
\[\diag(d_1,\ldots, d_r,0,\ldots,0) \in \CP \Leftrightarrow d_1 I_\ell,\ldots, d_r I_\ell \in \CP.\]
Thus, 
\[H_{\vt^t}(\diag(d_1,\ldots, d_r,0,\ldots,0)) = H_{\vt^t} (d_1 I_\ell) \cap \cdots \cap H_{\vt^t} (d_r I_\ell)\]
and the result follows.
\end{proof}

\begin{prop}\label{top_equal}
  The topologies $\CT_\s$ and $\CT_\s^\x$ on $X_{(A,\s)}$ are equal.
\end{prop}
\begin{proof}
  Use the homeomorphisms from Lemma~\ref{homeomorphisms}(1) to bring both topologies
  to $X_{(M_\ell(D), \vt^t)}$ on which they are equal by   Lemma~\ref{homeomorphisms}(2). 
\end{proof}

As a consequence of this proposition, we may use the sets $H_\s(a_1,\ldots a_k)$ as a basis of open sets for $\CT_\s=\CT_\s^\x$
 with $a_1,\ldots, a_k$   
in $\Sym(A,\s)$ or in $\Sym(A,\s)^\x$, whichever is more convenient for the problem at hand.

From now on, and unless
specified otherwise, we assume that $X_{(A,\s)}$ is equipped with the topology $\CT_\s$ and
that $\wt X_F$ is equipped with the topology $\CT_H$.

\begin{lemma}\label{m_P-cont}
  The map $X_F \rightarrow \Z$, $P \mapsto m_P(A,\s)$  \tu{(}cf. Definition~\ref{mp0}\tu{)} 
	is continuous.
\end{lemma}
\begin{proof}
  Since the five clopen sets defined in \eqref{def:subsets} cover
  $X_F$, it suffices to show that the map $P \mapsto m_P(A,\s)$ is
  continuous on each of them. Let $U$ be one of these clopen sets. We know from 
  Proposition~\ref{m_P=n_P} that $m_P(A,\s) = n_P$, and from the observation after \eqref{def:subsets}
   that the value of $n_P$ is constant on $U$.  The map is then constant on $U$ and
  therefore continuous.
\end{proof}

\begin{lemma}\label{ref-max}
  There are $c_1, \ldots, c_t \in \Sym(A,\s)^\x$ such that
  $X_{(A,\s)} = H_\s(c_1) \cup \cdots \cup H_\s(c_t)$.
\end{lemma}

\begin{proof}
  For $c \in \Sym(A,\s)^\x$ we define $H_\mm(c) := \{Q \in \wt X_F \mid
  \sign^\eta_Q \qf{c}_\s = m_Q(A,\s)\}$. By Lemma~\ref{m_P-cont} and \cite[Theorem~7.2]{A-U-Kneb}
	  the set $H_\mm(c)$ is
  clopen in $\wt X_F$.

  Let $P \in \wt X_F$. By definition there exists $c_P \in \Sym(A,\s)^\x$ such that $\sign^\eta_P
  \qf{c_P}_\s = m_P(A,\s)$, i.e., $P \in H_\mm(c_P)$. Therefore $\wt X_F = \bigcup_{c \in
  \Sym(A,\s)^\x} H_\mm(c)$ and by compactness of $\wt X_F$ we get $\wt X_F =
  H_\mm(c_1) \cup \cdots \cup H_\mm(c_r)$ for some $c_1, \ldots, c_r \in
  \Sym(A,\s)^\x$. It follows from  Theorem~\ref{positive=max} that 
  $X_{(A,\s)} = H_\s(c_1) \cup \cdots \cup H_\s(c_r) \cup H_\s(-c_1) \cup \cdots \cup H_\s(-c_r)$.
\end{proof}

\begin{remark}\label{rem:indep}
	It will be convenient in several proofs in this section to express the property $u \in
  \CP_F$ in terms of some element belonging to $\CP$. Observe that if $c \in
  \CP \setminus \{0\}$ and $u \in F \setminus \{0\}$, then $u \in \CP_F$ implies
  $uc \in \CP$, and $u \in -\CP_F$ implies $uc \in -\CP$. Therefore $u \in \CP_F$
  is equivalent to $uc \in \CP$. Moreover, the choice of $c$ is
  essentially independent of $\CP$. Indeed, 
	by Lemma~\ref{ref-max} we have $X_{(A,\s)} = H_\s(c_1) \cup \cdots \cup
  H_\s(c_t)$ for some $c_1, \ldots, c_t \in \Sym(A,\s)^\x$. 
  Thus, for $\CP \in H_\s(c_j)$ and $u
  \in F \setminus \{0\}$ we have $u \in \CP_F$ if and only if $uc_j \in \CP$.	
\end{remark}

\begin{thm}\label{thm:signp}
  Let $h$ be a nonsingular hermitian form over $(A,\s)$. Then the total signature map
	$\sign h : X_{(A,\s)} \rightarrow \Z,\ \CP\mapsto \sign_\CP h$  \tu{(}cf.
	 Definition~\ref{def_sign}\tu{)} 
	is continuous, where  $\Z$ is equipped with the discrete
  topology. 
\end{thm}

\begin{proof}
  By Lemma~\ref{morita-diag}, $\ell \x h \simeq \qf{a_1,\ldots, a_k}_\s$ for
  some $a_1, \ldots, a_k \in \Sym(A,\s)^\x$. Therefore $(\sign h)^{-1}(i) = (\sign
  \qf{a_1}_\s + \cdots + \sign \qf{a_k}_\s)^{-1}(\ell i)$ for any $i\in \Z$ and it suffices to
  show that $\sign \qf{a}_\s$ is continuous, for $a \in \Sym(A,\s)^\x$.

	By Lemma~\ref{ref-max} we have $X_{(A,\s)} = H_\s(c_1) \cup \cdots \cup
  H_\s(c_t)$ for some $c_1, \ldots, c_t \in \Sym(A,\s)^\x$ and so it is sufficient
	to show that the map $\sign \qf{a}_\s$ is continuous on
each $H_\s(c_j)$.  By Definition~\ref{def_sign}  we have $\sign_\CP \qf{a}_\s = i$ if and only if there exist $r,s,t\in \N$,
  $\beta_1,\ldots, \beta_t \in \CP_F\sm \{0\}$ and $a_1,\ldots, a_r, b_1,\ldots, b_s \in   \CP \cap 
	A^\x$ such that
  $({r-s})/({n_{\CP_F}t}) = i $ 
  and
  \[(n_{\CP_F})^2 \x \qf{\beta_1, \ldots, \beta_t}\ox \qf{a}_\s \simeq
  \qf{a_1,\ldots,a_r}_\s \perp \qf{-b_1,\ldots,-b_s}_\s.\]
  Therefore
  \begin{align*}
    (\sign & \qf{a}_\s)^{-1}(i) \cap H_\s(c_j)\\ 
    &= \{ \CP \in H_\s(c_j) \mid \exists 
    r,s,t \in \N,\ a_1,\ldots,a_r,b_1,\ldots,b_s \in \CP\cap A^\x, \\
    &\qquad \beta_1, \ldots, \beta_t \in \CP_F\sm\{0\} \text{ such that } (r-s)/n_{\CP_F}t = i,
		\text{ and } \\
    &\qquad (n_{\CP_F})^2 \x \qf{\beta_1, \ldots, \beta_t} \ox \qf{a}_\s \simeq
      \qf{a_1,\ldots,a_r}_\s \perp \qf{-b_1,\ldots,-b_s}_\s\} \\
    &=  \bigcup \{ H_\s(c_j,\beta_1c_j,\ldots,\beta_t c_j,a_1,\ldots,a_r,b_1,\ldots,b_s) \mid r,s,t
    \in \N,\  a_1,\ldots,a_r,\\ &\qquad b_1,\ldots,b_s \in \Sym(A,\s)^\x, \ 
     \beta_1, \ldots, \beta_t \in F^\x \text{ such that }
    (r-s)/n_{\CP_F}t = i, \\
    & \qquad \text{and } (n_{\CP_F})^2 \x \qf{\beta_1, \ldots, \beta_t} \ox \qf{a}_\s \simeq
      \qf{a_1,\ldots,a_r}_\s \perp \qf{-b_1,\ldots,-b_s}_\s\}
  \end{align*}
	is open for $\CT_\s$ (note that the second equality follows from Remark~\ref{rem:indep}, 
	since for $\CP \in H_\s(c_j)$ and $u \in F \setminus \{0\}$ 
	we have $u \in \CP_F$ if and only if $uc_j \in \CP$).	
\end{proof}

We now introduce the maps
\[
	\begin{aligned}
		\pi &: X_{(A,\s)}  \rightarrow \wt X_F, \ \CP \mapsto \CP_F,\\[5pt]
		 \xi &: \wt X_F \rightarrow X_{(A,\s)},\ P \mapsto \CC_P(\CM^\eta_P(A,\s)),
	\end{aligned}
\]
with reference to Theorem~\ref{positive=max} for the definition of $\xi$.

\begin{prop}\mbox{}\label{prop:pi-nice}
	\begin{enumerate}[$(1)$]
		\item  The map $\pi$ is open.
		\item  The map $\pi$ is continuous.
	\end{enumerate}
\end{prop}

\begin{proof}
	(1)  Let $a_1,\ldots, a_k \in \Sym(A,\s)^\x$. Then, using Theorem~\ref{positive=max},
  \[\pi(H_\s(a_1,\ldots, a_k)) = \{P \in \wt X_F \mid |\sign^\eta_P \qf{a_i}_\s| = m_P(A,\s),\
	 i=1,\ldots,k\},\]
  which is open since $P \mapsto m_P(A,\s)$  and  
  $P\mapsto  \sign^\eta_P \qf{a}_\s$ are continuous by Lemma~\ref{m_P-cont} and
	\cite[Theorem~7.2]{A-U-Kneb}.
	
	(2)   Let $u \in F$. We show that $\pi^{-1}(H(u))$ is open. Observe that
  \[\pi^{-1}(H(u)) = \{\CP \in X_{(A,\s)} \mid u \in \CP_F\}.\]
	Using Remark~\ref{rem:indep} we have
	for $\CP \in H_\s(c_j)$ that $u\in \CP_F$ if and only if  $\CP \in H_\s(uc_j)$.
		Therefore,
	\[
		\pi^{-1} (H(u)) = \bigcup_{j=1}^t \pi^{-1}(H(u)) \cap H_\s(c_j) = 
		\bigcup_{j=1}^t  H_\s(c_j, uc_j),
	\]
	which is in $\CT_\s$.
\end{proof}

\begin{prop}\mbox{}\label{prop:xi-nice}
	\begin{enumerate}[$(1)$]
		\item The maps $\xi$ and $-\xi$ are open.
		\item Let $b \in \Sym(A,\s)^\x$. Then $\xi^{-1}(H_\s(b))$ is
  clopen in $\CT_H$. In particular $\xi$ is continuous.
	\end{enumerate}
	
\end{prop}

\begin{proof}
	(1) Let $u_1, \ldots, u_k \in F$ and let $\ve \in \{-1,1\}$. We show that
  $\ve\xi(H(u_1, \ldots, u_k))$ is open for $\CT_\s$. By definition
  \[\ve\xi(H(u_1, \ldots, u_k)) = \{\ve\CC_P(\CM^\eta_P(A,\s)) \mid P \in H(u_1,
  \ldots, u_k)\}.\]
	By Lemma~\ref{ref-max}
	it suffices to show that $\ve \xi (H(u_1,\ldots, u_k)) \cap H_\s(c_i)$ is open for 
	$i=1,\ldots, t$.
	We have
	\[
		\begin{aligned}
			\ve \xi (H&(u_1,\ldots, u_k)) \cap H_\s(c_i) \\
			&=\{\CP \in X_{(A,\s)} \mid \CP= \ve \CC_P (\CM_P^\eta (A,\s)),\ P \in H(u_1,\ldots, u_k),\ 
			c_i\in \CP\}\\
			&=\{\CP \in X_{(A,\s)} \mid \sign_{\CP_F}^\eta \qf{c_i}_\s = \ve m_{\CP_F}(A,\s),\
			\CP \in H_\s(u_1c_i,\ldots, u_kc_i), \\ 
			& \hspace{25em} \CP \in H_\s(c_i)\}\\
			&\hspace{16.5em}\text{[by Remark~\ref{rem:indep} and Theorem~\ref{positive=max}]}\\
			&=\{\CP \in X_{(A,\s)} \mid \sign_{\CP_F}^\eta \qf{c_i}_\s = \ve m_{\CP_F}(A,\s)\}
			\cap H_\s(c_i,u_1c_i,\ldots, u_kc_i),
		\end{aligned}
	\]
	which is open. Indeed, the second set in the intersection is open by definition and the first
	set is open since the maps $\pi$ and 
	$P\mapsto \sign_P^\eta \qf{c_i}_\s$ are both continous, by Proposition~\ref{prop:pi-nice}(2)
	and \cite[Theorem~7.2]{A-U-Kneb}.

(2)  Let $b \in \Sym(A,\s)^\x$. Then $\xi^{-1}(H_\s(b)) = \{P \in \wt X_F \mid
  \sign^\eta_P \qf{b}_\s = m_P(A,\s)\}$, which is clopen since the map $X_F
  \rightarrow \Z$, $P \mapsto m_P(A,\s)$ is continuous by Lemma~\ref{m_P-cont}
  and the signature map is also continuous by \cite[Theorem~7.2]{A-U-Kneb}.
\end{proof}

\begin{prop}\label{pi-homeom}
  For every open set $U$ of $\wt X_F$, 
	\[
	\pi^{-1}(U) = (\pi^{-1}(U) \cap \im \xi)
  \cup (\pi^{-1}(U) \cap -\im \xi),
	\] 
	and $\pi \vert_{\pi^{-1}(U) \cap \ve
  \im \xi}$ is a homeomorphism from $\pi^{-1}(U) \cap \ve \im \xi$ to $U$,
  for $\ve \in \{-1,1\}$. 
  In particular, $\pi \vert_{\im\xi}$ and $\pi \vert_{-\im\xi}$ are homeomorphisms onto $\wt X_F$
	and 	$\pi$ is a covering map. 
\end{prop}

\begin{proof}
  The first statement is obvious since $X_{(A,\s)} = \im \xi \cup -\im \xi$. 
	That $\pi \vert_{\pi^{-1}(U) \cap \ve   \im \xi}$ is a homeomorphism follows
	from the fact that it is continous by Proposition~\ref{prop:pi-nice}(2) 
	and has inverse $\ve \xi$, which is also
  continuous by Proposition~\ref{prop:xi-nice}(2). The remaining statements follow immediately.
\end{proof}

\begin{cor}\label{double}
	$X_{(A,\s)}$ is homeomorphic to a disjoint union of two copies 	of $\wt X_F$.
\end{cor}

\begin{proof}
	The statement follows from Propositions~\ref{prop:xi-nice}(1) and  \ref{pi-homeom}.	
\end{proof}

\begin{prop}\mbox{}\label{H-compact}
  \begin{enumerate}[$(1)$]
    \item The sets $H_\s(a_1,\ldots, a_k)$ are compact  for any $k\in\N$ and  $a_1, \ldots, a_k \in
      \Sym(A,\s)$. 
    \item The compact open subsets of $\CT_\s$ are precisely the finite unions of sets of
      the form $H_\s(a_1, \ldots, a_k)$, for $k \in \N$ and $a_1, \ldots, a_k
      \in \Sym(A,\s)^\x$ \tu{(}or $a_1, \ldots, a_k \in \Sym(A,\s)$\tu{)}.
      \item $X_{(A,\s)}$ is compact.
  \end{enumerate}
\end{prop}

\begin{proof}
(1)      Assume that $H_\s(a_1, \ldots, a_k) = \bigcup_{i \in I} H_\s(\bar b_i)$ for
      some finite tuples $\bar b_i$ of elements of $\Sym(A,\s)$. Then $H_\s(a_1,
      \ldots, a_k) \cap \im \xi = \bigcup_{i \in I} H_\s(\bar b_i) \cap \im \xi$, and
      applying $\pi$ we get 
      \[\pi(H_\s(a_1, \ldots, a_k) \cap \im \xi) = \bigcup_{i \in I} \pi(H_\s(\bar
      b_i) \cap \im \xi),\]
      where $\pi(H_\s(a_1, \ldots, a_k) \cap \im \xi) = \xi^{-1}(H_\s(a_1, \ldots,
      a_k))$ and  $\pi(H_\s(\bar b_i) \cap \im \xi) = \xi^{-1}(H_\s(\bar b_i))$. 
			Furthermore, these two sets
      are clopen by Proposition~\ref{prop:xi-nice}(2), and since $\CT_H$ is compact there is a
      finite subset $I_1$ of $I$ such that
      \[\pi(H_\s(a_1, \ldots, a_k) \cap \im \xi )= \bigcup_{i \in I_1} \pi(H_\s(\bar
      b_i) \cap \im \xi).\]
      Similarly, with $-\xi$ instead of $\xi$, there is a finite subset $I_2$ of $I$
      such that
      \[\pi(H_\s(a_1, \ldots, a_k) \cap -\im \xi )= \bigcup_{i \in I_2} \pi(H_\s(\bar
      b_i) \cap -\im \xi).\]
      It follows that
      \[H_\s(a_1, \ldots, a_k) = \bigcup_{i \in I_1 \cup I_2} H_\s(\bar b_i),\]
      proving the result.
			
(2) If $U$ is open in $\CT_\s$, then $U$ is a union of open sets of the
      form $H_\s(a_1,\ldots,a_k)$, and if $U$ is compact, this union can be taken
      finite (with $a_1, \ldots, a_k$ in $\Sym(A,\s)^\x$ or $\Sym(A,\s)$, cf.
      Proposition~\ref{top_equal}). The other part of the statement is immediate
      since the sets $H_\s(a_1, \ldots, a_k)$ are compact in $\CT_\s$ by (1).
			
(3) The statement follows immediately from (1) and Lemma~\ref{ref-max}.
\end{proof}

\begin{thm}\label{thm:spectral}
	$\CT_\s$ is spectral.	
\end{thm}
\begin{proof}
  By \cite[Proposition~4~(i)$\Leftrightarrow$(v)]{Hochster}, 
	and in view of Proposition~\ref{H-compact}, to
  show that $\CT_\s$ is spectral it suffices to show that $\CT_\s$ is $T_0$ and
  that $(\CT_\s)_\patch$ is compact.

  $\CT_\s$ is $T_0$: Let $\CP_1 \not = \CP_2 \in X_{(A,\s)}$. Without loss of generality we may
	assume that there exists $a \in \CP_1 \setminus \CP_2$. Then $\CP_1 \in H_\s(a)$ and
	$\CP_2 \not\in H_\s(a)$. 

  $(\CT_\s)_\patch$ is compact:
	Since $\CT_\s$ is compact, the sets of the form $U \setminus V$ with $U,V$ compact
  open in $\CT_\s$, form a subbasis of open sets of $(\CT_\s)_\patch$. Assume
  \[X_{(A,\s)} = \bigcup_{i \in I} (\bigcap_{j \in J_i} U_{i,j} \setminus V_{i,j}),\]
  where each set $J_i$ is finite, and the sets $U_{i,j}$, $V_{i,j}$ are compact open in
  $\CT_\s$. Then
  \[\im \xi =  \bigcup_{i \in I} (\bigcap_{j \in J_i} (U_{i,j} \cap \im \xi)
  \setminus (V_{i,j} \cap \im \xi)),\]
  where each $U_{i,j} \cap \im \xi$, $V_{i,j} \cap \im \xi$ is compact open in $\im \xi$
	since both $\im\xi$ and $-\im\xi$ are clopen by Proposition~\ref{prop:xi-nice}(1). 
	Applying $\pi$, we get 
  \[\wt X_F = \bigcup_{i \in I} (\bigcap_{j \in J_i} \pi(U_{i,j} \cap \im \xi)
  \setminus \pi(V_{i,j} \cap \im \xi)),\]
  with each $\pi(U_{i,j} \cap \im \xi)$, $\pi(V_{i,j} \cap \im \xi)$ compact open in
  $\wt X_F$ by Proposition~\ref{pi-homeom}, 
	so in particular clopen, and thus each $\bigcap_{j \in J_i}
  \pi(U_{i,j} \cap \im \xi) \setminus \pi(V_{i,j} \cap \im \xi))$ is open in
  $\wt X_F$. By compactness there is a finite subset $I_1$ of $I$ such that
  \[\wt X_F = \bigcup_{i \in I_1} ( \bigcap_{j \in J_i}
  \pi(U_{i,j} \cap \im \xi) \setminus \pi(V_{i,j} \cap \im \xi)).\]
  Reasoning similarly with $-\xi$, we obtain a finite subset $I_2$ of $I$ such
  that
  \[\wt X_F = \bigcup_{i \in I_2}( \bigcap_{j \in J_i}
  \pi(U_{i,j} \cap -\im \xi) \setminus \pi(V_{i,j} \cap -\im \xi)).\]
  Therefore
  \[X_{(A,\s)} = \bigcup_{i \in I_1 \cup I_2} (\bigcap_{j \in J_i}
  U_{i,j} \setminus V_{i,j}),\]
  proving that $(\CT_\s)_\patch$ is compact.
\end{proof}

\begin{prop}\label{MMMMorita}
  Let $(A,\s)$ and $(B,\tau)$ be two Morita equivalent 
  $F$-algebras with involution. 
  With notation as in Theorem~\ref{morco}, the map 
  $\mor_*$ is a homeomorphism from $(X_{(A,\s)}, \CT_\s)$ to $(X_{(B,\tau)}, \CT_\tau)$.
\end{prop}

\begin{proof} 
Let $b_1, \ldots, b_k \in
\Sym(B,\tau)^\x$. We have $\mor_*(\CP) \in H_\tau(b_1,\ldots, b_k)$ if and
only if $b_1,\ldots, b_k \in \mor_*(\CP)$ if and only if there are $a_1, \ldots,
a_r \in \CP$ (see Theorem~\ref{morco}) such that $b_1, \ldots, b_k \in
D_{(B,\tau)}\mor( \qf{a_1,\ldots, a_r}_\s)$. Therefore
\begin{align*}
  (\mor_*)^{-1}&(H_\tau(b_1,\ldots,b_k)) \\
  &= \{\CP \in X_{(A,\s)} \mid \exists r \in   \N \ \exists a_1, \ldots, a_r \in \CP
	\text{ such that } \\ 
   &\hspace{15.8em} b_1, \ldots, b_k \in D_{(B,\tau)} \mor( \qf{a_1,\ldots, a_r}_\s)\}\\
  &= \{\CP \in X_{(A,\s)} \mid \exists r \in \N \ \exists a_1, \ldots, a_r \in
    \Sym(A,\s)\text { such that} \\
  &\qquad\qquad\qquad \CP \in H_\s(a_1,\ldots, a_r) \text{ and } b_1, \ldots, b_k \in
    D_{(B,\tau)}\mor(\qf{a_1,\ldots, a_r}_\s)\}\\ 
  &= \bigcup \{H_\s(a_1,\ldots, a_r) \mid r \in \N, \ a_1, \ldots, a_r
    \in \Sym(A,\s),\\
  &\hspace{15.8em} b_1, \ldots, b_k \in D_{(B,\tau)}\mor( \qf{a_1,\ldots, a_r}_\s)\},
\end{align*}
which is a union of open sets in the topology $\CT_\s$. The fact that
$(\mor_*)^{-1}$ is also continuous is obtained as above, using that $(\mor_*)^{-1} = (\mor^{-1})_*$ as observed 
in the proof of Theorem~\ref{morco}.
\end{proof}

\section*{Acknowledgement}

We are very grateful to the referee, in particular for some suggestions that led to a significantly
improved final section.


\def\cprime{$'$}


\begin{thebibliography}{10}

\bibitem{Albert55}
Albert, A.A.:
\newblock On involutorial algebras.
\newblock {\em Proc. Nat. Acad. Sci. U.S.A.}, 41:480--482, 1955.

\bibitem{A-U-Kneb}
Astier, V., Unger, T.:
\newblock Signatures of hermitian forms and the {K}nebusch trace formula.
\newblock {\em Math. Ann.}, 358(3-4):925--947, 2014.


\bibitem{A-U-prime}
Astier, V., Unger, T.:
\newblock Signatures of hermitian forms and ``prime ideals'' of {W}itt groups.
\newblock {\em Adv. Math.}, 285:497--514, 2015.

\bibitem{A-U-PS}
Astier, V., Unger, T.:
\newblock Signatures of hermitian forms, positivity, and an answer to a
  question of procesi and schacher.
\newblock{\em J. Algebra}, 508:339--363, 2018.

\bibitem{A-U-stab}
Astier, V., Unger, T.:
Stability index for algebras with involution, 
 \emph{Contemporary Mathematics} 697:41--50, 2017. 

\bibitem{A-U-gauges}
Astier, V., Unger, T.:
Positive cones and gauges on algebras with involution,
{\em \url{http://arxiv.org/abs/1806.05489}}, 2019.




\bibitem{Charnes-Cooper}
Charnes, A., Cooper, W.W.:
\newblock The strong {M}inkowski-{F}arkas-{W}eyl theorem for vector spaces over
  ordered fields.
\newblock {\em Proc. Nat. Acad. Sci. U.S.A.}, 44:914--916, 1958.

\bibitem{Craven-1995}
Craven, T.C.:
\newblock Orderings, valuations, and {H}ermitian forms over {$*$}-fields.
\newblock In {\em {$K$}-theory and algebraic geometry: connections with
  quadratic forms and division algebras ({S}anta {B}arbara, {CA}, 1992)},
  volume~58 of {\em Proc. Sympos. Pure Math.}, pages 149--160. Amer. Math.
  Soc., Providence, RI, 1995.


\bibitem{DST}
Dickmann, M., Schwartz, N., Tressl, M.:
\newblock \emph{Spectral spaces}, 
\newblock New Mathematical Monographs, 35. Cambridge University Press, Cambridge, 2019.


\bibitem{ELP}
Elman, R.,  Lam, T.Y., Prestel, A.:
\newblock On some {H}asse principles over formally real fields.
\newblock {\em Math. Z.}, 134:291--301, 1973.

\bibitem{Hochster}
Hochster, M.:
\newblock Prime ideal structure in commutative rings.
\newblock {\em Trans. Amer. Math. Soc.}, 142:43--60, 1969.

\bibitem{Knus}
Knus, M.-A.:
\newblock {\em Quadratic and {H}ermitian forms over rings}, volume 294 of {\em
  Grundlehren der Mathematischen Wissenschaften}.
\newblock Springer-Verlag, Berlin, 1991.

\bibitem{BOI}
Knus, M.-A., Merkurjev, A., Rost, M.,   Tignol, J.-P.:
\newblock {\em The book of involutions}, volume~44 of {\em American
  Mathematical Society Colloquium Publications}.
\newblock American Mathematical Society, Providence, RI, 1998.


\bibitem{Lam}
Lam, T.Y.:
\newblock {\em Introduction to quadratic forms over fields}, volume~67 of {\em
  Graduate Studies in Mathematics}.
\newblock American Mathematical Society, Providence, RI, 2005.

\bibitem{Lee-1949}
Lee, H.C.:
\newblock Eigenvalues and canonical forms of matrices with quaternion
  coefficients.
\newblock {\em Proc. Roy. Irish Acad. Sect. A.}, 52:253--260, 1949.


\bibitem{LU1}
Lewis, D.W., Unger, T.:
\newblock A local-global principle for algebras with involution and {H}ermitian
  forms.
\newblock {\em Math. Z.}, 244(3):469--477, 2003.

\bibitem{LU2}
Lewis, D.W., Unger, T.:
\newblock Hermitian {M}orita theory: a matrix approach.
\newblock {\em Irish Math. Soc. Bull.}, (62):37--41, 2008.

\bibitem{Marshall_real_spectra}
Marshall, M.A.:
\newblock {\em Spaces of orderings and abstract real spectra}, volume 1636 of
  {\em Lecture Notes in Mathematics}.
\newblock Springer-Verlag, Berlin, 1996.

\bibitem{Plat-Yan}
Platonov, V.P., Yanchevskii, V.I.:
\newblock Finite-dimensional division algebras.
\newblock In {\em Algebra, {IX}}, volume~77 of {\em Encyclopaedia Math. Sci.},
  pages 121--239. Springer, Berlin, 1995.
\newblock Translated from the Russian by P. M. Cohn.



\bibitem{Prestel73}
Prestel, A.:
\newblock Quadratische {S}emi-{O}rdnungen und quadratische {F}ormen.
\newblock {\em Math. Z.}, 133:319--342, 1973.

\bibitem{Prestel84}
Prestel, A.:
\newblock {\em Lectures on formally real fields}, volume 1093 of {\em Lecture
  Notes in Mathematics}.
\newblock Springer-Verlag, Berlin, 1984.

\bibitem{PD2001}
Prestel, A., Delzell, C.N.:
\newblock {\em Positive polynomials}.
\newblock Springer Monographs in Mathematics. Springer-Verlag, Berlin, 2001.



\bibitem{P-S}
Procesi, C., Schacher, M.:
\newblock A non-commutative real {N}ullstellensatz and {H}ilbert's 17th
  problem.
\newblock {\em Ann. of Math. (2)}, 104(3):395--406, 1976.

\bibitem{Sch}
Scharlau, W.:
\newblock {\em Quadratic and {H}ermitian forms}, volume 270 of {\em Grundlehren
  der Mathematischen Wissenschaften}.
\newblock Springer-Verlag, Berlin, 1985.


\bibitem{T-W-2011}
Tignol, J.-P., Wadsworth, A.R.:
\newblock Valuations on algebras with involution.
\newblock {\em Math. Ann.}, 351(1):109--148, 2011.


\bibitem{Weil}
Weil, A.:
\newblock Algebras with involutions and the classical groups.
\newblock {\em J. Indian Math. Soc. (N.S.)}, 24:589--623 (1961), 1960.

\end{thebibliography}
\end{document}